\documentclass{svjour3}                     
\smartqed  

\usepackage{setspace}
\usepackage{algorithm}
\usepackage{algorithmic}

\usepackage{amssymb}
\usepackage{booktabs}
\usepackage{graphicx,fancyhdr}
\usepackage{graphics,color}
\usepackage{subfigure}
\usepackage{booktabs}
\usepackage{graphicx}
\usepackage{amsmath}
\usepackage{graphicx}
\usepackage{amsmath}
\usepackage{epstopdf} 
\usepackage{bbding}
\usepackage{mathptmx}
\usepackage{algorithm}
\usepackage{algorithmic}
\usepackage{color}
\numberwithin{equation}{section}
\numberwithin{theorem}{section}
\numberwithin{lemma}{section}
\numberwithin{remark}{section}

\usepackage{mathptmx}
\begin{document}

\title{A second-order accurate scheme for two-dimensional  space fractional   diffusion  equations with  time   Caputo-Fabrizio fractional derivative}



\author{Jiankang Shi         \and
         Minghua Chen
}

\institute{
J. Shi    \and     M. Chen   (\Envelope)    \\
School of Mathematics and Statistics, Gansu Key Laboratory of Applied Mathematics and Complex Systems, Lanzhou University, Lanzhou 730000, P.R. China\\
email:  shijk17@lzu.edu.cn; chenmh@lzu.edu.cn; }
\date{Received: date / Accepted: date}

\maketitle

\begin{abstract}
We provide and analyze a  second order   scheme for the model describing
the functional distributions of particles performing anomalous motion with exponential Debye pattern  and no-time-taking jumps eliminated,
and  power-law  jump length. The model is derived in [M. Chen, J. Shi, W. Deng, arXiv:1809.03263], being called the space fractional   diffusion  equation with the time   Caputo-Fabrizio fractional derivative.
The designed schemes are unconditionally stable and have the second order global truncation error with  the nonzero initial condition,
being theoretically proved and numerically verified by two methods (a prior estimate with $L^2$-norm and mathematical induction with $l_\infty$ norm). Moreover,  the optimal estimates are obtained.

\end{abstract}

\section{Introduction}

The  Caputo-Fabrizio fractional derivative  \cite{Caputo:15}  has been used to model a variety of applied scientific phenomena, such as control systems \cite{Alkahtani:16},  physics \cite{Abdulhameed:19,Atanackovie:18,Atangana:18},
 medicine \cite{Ullah:18}, fluid dynamics \cite{Abdulhameed:17,Asif:18,Mahsud:18,Shah:16}.
It is able to describe  the material heterogeneities and the fluctuations of   different scales.
Based on the continuous time random walk  with exponential Debye pattern  and no-time-taking jumps eliminated,
and  power-law  jump length, i.e., taking the waiting time distribution is
$\sigma\left(1+\sigma-\delta (t) \right)\exp(-\sigma t)$, $\sigma=\gamma/(1-\gamma)$ and the jump length distribution is $|x|^{-1-\alpha}$,
we derive  the following  space fractional   diffusion  equation with the time   Caputo-Fabrizio fractional operator \cite{CSD:18}
\begin{equation} \label{1.1}
\left\{ \begin{array}
 {l@{\quad} l}
 {^{CF}_0 D^\gamma _t}u(x,y,t)= \displaystyle \frac{\partial ^{\alpha}u(x,y,t)}{\partial |x|^{\alpha}}+\frac{\partial ^{\beta}u(x,y,t)}{\partial |y|^{\beta}}+f(x,y,t),\\
  u(x,y,0)=u_0(x,y) ~~~~\, {\rm for}~~~ (x,y) \in \Omega,\\
   u(x,y,t)=0 ~~~ {\rm for}~~~ (x,y,t) \in \partial \Omega \times [0,T],
 \end{array}
 \right.
\end{equation}
on a finite rectangular  domain $\Omega=(0,x_R) \times (0,y_R)$ and $ 0< t \leq T$. The  Caputo-Fabrizio fractional derivative,  for $0<\gamma<1$,  is defined by  \cite{Caputo:15,CSD:18}

\begin{equation} \label{1.2}
\begin{split}
{^{CF}_0 D^\gamma _t}u(t)
& = \frac{1}{1-\gamma}\int_{0}^t{u' \left(s\right) e^{-\frac{\gamma}{1-\gamma}(t-s)}} ds
 = \frac{1}{1-\gamma}\int_{0}^t{u' \left(s\right) e^{-\sigma(t-s)}} ds,~~\sigma=\frac{\gamma}{1-\gamma}.
\end{split}
\end{equation}
 The Riesz fractional derivative is given in  \cite{Podlubny:99}

\begin{equation} \label{1.3}
\frac{\partial^\alpha u(x,t) }{\partial |x|^\alpha}= \kappa_{\alpha} \left(_0 D ^{\alpha}_x +_x D ^{\alpha}_{x_R}\right)u(x,t),~~\kappa_{\alpha}=-\frac{1}{2\cos(\alpha\pi/2)}>0, ~1<\alpha<2
\end{equation}
with

\begin{equation*}
 _{0} D_x ^{\alpha} u(x,t) =
  \frac{1}{\Gamma(2-\alpha)}  \frac{\partial^2}{\partial x^2} \int_{0}^x {\left(x-\xi\right)^{1-\alpha}} {u(\xi,t)}d\xi,
\end{equation*}
\begin{equation*}
 _{x} D_{x_R} ^{\alpha} u(x,t) =
  \frac{1}{\Gamma(2-\alpha)}  \frac{\partial^2}{\partial x^2} \int_{x}^b {\left(\xi-x \right)^{1-\alpha}} {u(\xi,t)}d\xi.
\end{equation*}

In recently years, numerical method for solving the Caputo-Fabrizio fractional derivative  \cite{Caputo:15} is experiencing rapid development.
For example, the new operational matrix together with Tau method has been used to solve the equation with Caputo-Fabrizio operator \cite{Loh:18}.
Numerical approach of Fokker-Planck equation with Caputo-Fabrizio fractional derivative  is discussed in \cite{Firoozjaee:18}.
A second-order  Crank-Nicolson scheme \cite{Zhao:14} of the time fractional Caputo-Fabrizio derivative with $1<\gamma<2$ is proposed in \cite{Liu:17};
and the stability analysis of the numerical scheme for the groundwater model with Caputo-Fabrizio operator are proven in \cite{Djida:17}.
 Using the Lubich's operator \cite{Lubich:86} and the discretized fractional substantial calculus \cite{CD:14,CD:15}, the stability of the second-order scheme for Caputo-Fabrizio fractional equation are  proved \cite{CSD:18} by  a priori estimate \cite{Ji:15} under the zero initial condition.
Based on the idea of L1 formula \cite{Lin:07,Oldham:74}, a numerical approximation to the Caputo-Fabrizio derivative by using a linear interpolation is provided \cite{Akman:18};
and the first-order convergence analyse are given in \cite{Atangana:16}. It seems that achieving a second-order accurate scheme for L1 formula is not an easy task.
This paper focused on providing effective and a second-order accurate scheme for (\ref{1.1}). Under the nonzero initial condition,
the numerical stability and convergence of the  L1 scheme with second-order accuracy are theoretically proved by two ways (a prior estimate with $L^2$-norm and mathematical induction with $L_\infty$ norm); and the optimal estimates are obtained.

The paper is organized as follows. In Section 2, we provide the approximation scheme to  the Caputo-Fabrizio fractional derivative, and the full discretization of (\ref{1.1}).
 In Section 3, the unconditionally stability and convergence of the numerical schemes are proved in detail. In Section 4, we use the numerical example to verify the unconditionally stability and the convergence order of the difference schemes.
 Finally, we conclude the paper with some remarks.

\section{Discretization schemes}
Let the mesh points $x_i=i\Delta x$, $i=0,1,2,\cdots,N_x$, $y_j=j\Delta y$, $j=0,1,2,\cdots,N_y$ and $t_n=n\tau$, $n=0,1,2,\cdots,N$,
where $\Delta x=\frac{x_{R}}{N_x}$, $\Delta y=\frac{y_{R}}{N_y}$ and $\tau=\frac{T}{N}$ are the uniform  space stepsize and  time steplength, respectively.
Denote  $u^{n}_{i,j}$ as the numerical approximation to $u(x_i.y_j,t_n)$ and $f^{n}_{i,j}=f(x_i,y_j,t_n)$.
\subsection{Discretized Caputo-Fabrizio fractional derivative} \label{SubSec:2.1}
In this subsection we provide a  second-order discretization L1 formula for  the  Caputo-Fabrizio fractional derivative,
although there is less than the second-order convergence for  the  Caputo fractional derivative \cite{Lin:07}.
\begin{lemma}\label{lemma2.1}
 Let $0<\gamma<1$  with $\sigma = \frac {\gamma}{1-\gamma}$. Let $u(t)$ be sufficiently smooth for $t\geq 0$. Then
 \begin{equation*}
  {^{CF}_0 D^\gamma _{t}}u(t_{n}) = \frac{1}{1-\gamma} \sum^{n}_{k=1} \frac{u(t_{k})-u(t_{k-1})}{\sigma\tau} e^{-\sigma(n-k)\tau} (1-e^{-\sigma\tau}) +\mathcal{O}\left(\tau^2\right).
\end{equation*}
\end{lemma}
\begin{proof}
We can rewrite (\ref{1.2}) as
\begin{equation*}
\begin{split}
{^{CF}_0 D^\gamma _{t}}u(t_{n})
& = \frac{1}{1-\gamma} \int_{0}^{t_{n}}{u' \left(s\right) e^{-\sigma(t_{n}-s)}} ds  = \frac{1}{1-\gamma} \sum^{n}_{k=1} \int_{t_{k-1}}^{t_{k}} {u' \left(s\right) e^{-\sigma(t_{n}-s)}} ds \\
& = \frac{1}{1-\gamma} \sum^{n}_{k=1} \frac{u(t_{k})-u(t_{k-1})}{\tau} \int_{t_{k-1}}^{t_{k}}{ e^{-\sigma(t_{n}-s)}} ds +r^{n}_{\tau}\\
& = \frac{1}{1-\gamma} \sum^{n}_{k=1} \frac{u(t_{k})-u(t_{k-1})}{\sigma\tau} e^{-\sigma(n-k)\tau} (1-e^{-\sigma\tau}) +r^{n}_{\tau}.
\end{split}
\end{equation*}
Here $r^{n}_{\tau}= -I_1-I_2$ with
\begin{equation} \label{2.1}
\begin{split}
&I_1 = \frac{1}{1-\gamma} \sum^{n}_{k=1} \int^{t_{k}}_{t_{k-1}}{\frac{u''\left(s\right)(t_{k}+t_{k-1}-2s)}{2} e^{-\sigma (t_{n}-s)}}ds\\
&I_2= \frac{1}{1-\gamma} \sum^{n}_{k=1} \int^{t_{k}}_{t_{k-1}}{\frac{(t_{k}-s)^3 u'''(\zeta_{1k})-(t_{k-1}-s)^3u'''(\zeta_{2k})}{6 \tau} e^{-\sigma (t_{n}-s)}}ds,
\end{split}
\end{equation}
and $\zeta_{1k},\zeta_{2k} \in (t_{k-1},t_{k})$.

Next we shall  estimate $r^{n}_{\tau}= \mathcal{O}\left(\tau^2\right)$.
According to the first equation of (\ref{2.1}), the first mean value theorem for definite integrals and Taylor series expansion, we have
\begin{equation*}
\begin{split}
I_1
=& \frac{1}{1-\gamma} \sum^{n}_{k=1} \int^{t_{k-\frac{1}{2}}}_{t_{k-1}}{\frac{u''\left(s\right)(t_{k}+t_{k-1}-2s)}{2} e^{-\sigma (t_{n}-s)}}ds\\
 &+\frac{1}{1-\gamma}\sum^{n}_{k=1}\int^{t_{k}}_{t_{k-\frac{1}{2}}}{\frac{u''\left(s\right)(t_{k}+t_{k-1}-2s)}{2} e^{-\sigma (t_{n}-s)}}ds \\
=&\frac{1}{1-\gamma}\sum^{n}_{k=1} \frac{u''(\eta_{1k})}{2} \int^{t_{k-\frac{1}{2}}}_{t_{k-1}}{(t_{k}+t_{k-1}-2s) e^{-\sigma (t_{n}-s)}}ds \\
 &+\frac{1}{1-\gamma}\sum^{n}_{k=1}\frac{u''(\eta_{2k})}{2}\int^{t_{k}}_{t_{k-\frac{1}{2}}}{(t_{k}+t_{k-1}-2s)e^{-\sigma (t_{n}-s)}}ds \\
=&I_{11}+I_{12},
\end{split}
\end{equation*}
where
\begin{equation}\label{2.2}
\begin{split}
I_{11}
=& \frac{1}{2(1-\gamma)} \sum^{n}_{k=1} u''(\eta_{1k}) \int^{t_{k}}_{t_{k-1}}{(t_{k}+t_{k-1}-2s) e^{-\sigma (t_{n}-s)}}ds\\
I_{12}
=&\frac{1}{2(1-\gamma)} \sum^{n}_{k=1}(\eta_{2k}-\eta_{1k})u'''(\eta_{3k}) \int^{t_{k}}_{t_{k-\frac{1}{2}}}{(t_{k}+t_{k-1}-2s) e^{-\sigma (t_{n}-s)}}ds
\end{split}
\end{equation}
with $\eta_{1k},\eta_{2k},\eta_{3k} \in (t_{k-1},t_{k})$.
From the definite integrals of  (\ref{2.2}), there exists
\begin{equation*}
\begin{split}
I_{3k}
& = \int^{t_{k}}_{t_{k-1}}{(t_{k}+t_{k-1}-2s) e^{-\sigma (t_{n}-s)}}ds\\
& = e^{-\sigma (n-k) \tau} \left[ -\frac{\tau}{\sigma}\left( 1+ e^{-\sigma  \tau}\right)   + \frac{2}{\sigma^2} \left( 1- e^{-\sigma  \tau}\right) \right]
 = e^{-\sigma (n-k)\tau} \left[   - \frac{\sigma}{6} \tau^3 +\mathcal{O}\left(\tau^4\right) \right];
\end{split}
\end{equation*}
and
\begin{equation*}
\begin{split}
I_{4k}
=&  \int^{t_{k}}_{t_{k-\frac{1}{2}}}{(t_{k}+t_{k-1}-2s) e^{-\sigma (t_{n}-s)}}ds\\
 = &e^{-\sigma(n-k)\tau} \left[ \frac{2}{\sigma^2} \left( 1-e^{-\frac{1}{2}\sigma\tau} \right) - \frac{\tau}{\sigma} \right]
  =e^{-\sigma(n-k)\tau} \left[-\frac{1}{4}\tau^2  +\mathcal{O}\left(\tau^3\right)\right].
\end{split}
\end{equation*}
Therefore, we have
\begin{equation*}
\begin{split}
I_{11} &= \frac{1}{2(1-\gamma)} \sum^{n}_{k=1} u''(\eta_{1k})  I_{3k}=\mathcal{O}\left(\tau^2\right),\\
I_{12}& = \frac{1}{2(1-\gamma)} \sum^{n}_{k=1}(\eta_{2k}-\eta_{1k})u'''(\eta_{3k})  I_{4k}=\mathcal{O}\left(\tau^2\right),
\end{split}
\end{equation*}
where we use $\eta_{2k}-\eta_{1k}=\mathcal{O}(\tau)$. From the above equations, we get
\begin{equation*}
I_1 = I_{11} + I_{12} =\mathcal{O}\left(\tau^2\right).
\end{equation*}
Since
\begin{equation*}
I_2 = I_{21} + I_{22}
\end{equation*}
with
\begin{equation*}
\begin{split}
I_{21}
& =  \frac{1}{6 \tau(1-\gamma)} \sum^{n}_{k=1} u'''(\zeta_{1k})  \int^{t_{k}}_{t_{k-1}}{(t_{k}-s)^3 e^{-\sigma (t_{n}-s)}}  ds,\\
I_{22}
&=- \frac{1}{6 \tau(1-\gamma)} \sum^{n}_{k=1} u'''(\zeta_{2k})  \int^{t_{k}}_{t_{k-1}}{(t_{k-1}-s)^3 e^{-\sigma (t_{n}-s)}}ds.
\end{split}
\end{equation*}
It is easy to check that
\begin{equation*}
\begin{split}
I_{5k}
& = \int^{t_{k}}_{t_{k-1}}{(t_{k}-s)^3 e^{-\sigma (t_{n}-s)}}ds
 = e^{-\sigma ({n-k+1})\tau} \left[ \frac{6}{\sigma^4}(e^{\sigma\tau} -1 ) - \frac{6\tau}{\sigma^3} -\frac{3\tau^2}{\sigma^2} -\frac{\tau^3}{\sigma}  \right]\\
& = e^{-\sigma ({n-k+1})\tau} \left[ \frac{1}{4}\tau^4 + \mathcal{O}\left(\tau^5\right) \right],
  \end{split}
\end{equation*}
and
\begin{equation*}
\begin{split}
 I_{6k}
& = \int^{t_{k}}_{t_{k-1}}{(t_{k-1}-s)^3 e^{-\sigma (t_{n}-s)}}ds
 = e^{-\sigma (n-k) \tau} \left[ \frac{6}{\sigma^4}(1 - e^{ -\sigma\tau} ) - \frac{6\tau}{\sigma^3} +\frac{3\tau^2}{\sigma^2} -\frac{\tau^3}{\sigma}  \right]\\
 & = e^{-\sigma (n-k) \tau} \left[- \frac{1}{4}\tau^4 + \mathcal{O}\left(\tau^5\right)\right].
  \end{split}
\end{equation*}
It yields
\begin{equation*}
  I_2 = I_{21}+I_{22} =  \mathcal{O}\left(\tau^2\right).
\end{equation*}
According to the above equations, we have
\begin{equation*}
  |r^{n}_{\tau}| \leq |I_1|+|I_2| = \mathcal{O}\left(\tau^2\right).
\end{equation*}
The proof is completed.
\end{proof}
\subsection{Derivation of numerical schemes for 1D}
Consider  the one-dimensional time-space Caputo-Riesz fractional   diffusion equation
\begin{equation}\label{2.0016}
{^{CF}_0 D^\gamma _t}u(x,t)=\frac{\partial ^{\alpha}u(x,t)}{\partial |x|^{\alpha}}+f(x,t).
\end{equation}
To discrete the Riesz fractional derivative for $1<\alpha<2$, we notice that the approximation operator of (\ref{1.3}) is given in  \cite{Chen:14}
\begin{equation*}
\begin{split}
& \delta_{\alpha,+x}u^{n}_{i}:=\frac{1}{\Gamma(4-\alpha)(\Delta x)^{\alpha}} \sum^{i+1}_{m=0}{g^{\alpha}_{m}u^{n}_{i-m+1}} = \frac{1}{\Gamma(4-\alpha)(\Delta x)^{\alpha}} \sum^{i+1}_{m=0}{g^{\alpha}_{i-m+1}u^{n}_{m}};\\
&\delta_{\alpha,-x}u^{n}_{i}:=\frac{1}{\Gamma(4-\alpha)(\Delta x)^{\alpha}} \sum^{N_x-i+1}_{m=0}{g^{\alpha}_{m}u^{n}_{i+m-1}} = \frac{1}{\Gamma(4-\alpha)(\Delta x)^{\alpha}} \sum^{N_x}_{m=i-1}{g^{\alpha}_{m-i+1}u^{n}_{m}},
\end{split}
\end{equation*}
and there exists
\begin{equation*}
  _{0} D_x ^{\alpha} u(x,t) \big|_{x=x_i} = \delta_{\alpha,+x}u^{n}_{i} + \mathcal{O}\left(\left(\Delta x\right)^2\right)~~{\rm and}~~ _{x} D_{x_R} ^{\alpha} u(x,t) \big|_{x=x_{i}}
  = \delta_{\alpha,-x}u^{n}_{i} + \mathcal{O}\left(\left(\Delta x\right)^2\right)
\end{equation*}
with
\begin{equation} \label{2.3}
  g^{\alpha}_{m} = \left\{
\begin{array}{lcl}
1,                                                                                             &     &{m=0},\\
-4+2^{3-\alpha},                                                                               &     &{m=1},\\
6-2^{5-\alpha}+3^{3-\alpha},                                                                   &     &{m=2},\\
(m+1)^{3-\alpha}-4m^{3-\alpha}+6(m-1)^{3-\alpha}-4(m-2)^{3-\alpha}+(m-3)^{3-\alpha},           &     &{m\geq3}.
\end{array} \right.
\end{equation}
Hence, the discrete scheme of the Riesz fractional derivative is
\begin{equation} \label{2.4}
  \begin{split}
    \frac{\partial^\alpha u (x,t)}{\partial |x|^\alpha} \bigg|_{x=x_{i}}
    & = \kappa_{\alpha} \left(_0 D ^{\alpha}_x +_x D ^{\alpha}_{x_R}\right)u(x,t) \big|_{x=x_{i}}
     = \kappa_{\alpha} \left( \delta_{\alpha,+x} + \delta_{\alpha,-x} \right)u^{n}_{i} + \mathcal{O}\left(\left(\Delta x\right)^2\right) \\
    & = \frac{ \kappa_{\alpha} }{ \Gamma(4-\alpha)(\Delta x)^{\alpha} } \sum^{N_x}_{m=0}{ \widetilde{g}^{\alpha}_{i,m} u^{n}_{m}}+\mathcal{O}\left(\left(\Delta x\right)^2\right),
  \end{split}
\end{equation}
where
\begin{equation} \label{2.5}
 g^{\alpha}_{i,m} = \left\{
  \begin{array}{lcl}
    g^{\alpha}_{i-m+1},              &     &{m<i-1},\\
    g^{\alpha}_{0}+g^{\alpha}_{2},   &     &{m=i-1},\\
    2g^{\alpha}_{1},                 &     &{m=i},  \\
    g^{\alpha}_{0}+g^{\alpha}_{2},   &     &{m=i+1},\\
    g^{\alpha}_{m-i+1},              &     &{m>i+1},
  \end{array} \right.
\end{equation}
with $i=1,2,\cdots,N_x-1$, $u^{n}_{0} $ and $u^{n}_{N_x}$ are the  boundary conditions.

Taking $u=[u(x_1),u(x_2),\cdots,u(x_{N_x-1})]^T$, then
\begin{equation*}
  \left[\sum^{N_x}_{m=0}{g^{\alpha}_{1,m} u(x_{m})} , \sum^{N_x}_{m=0}{g^{\alpha}_{2,m} u(x_{m})} , \cdots, \sum^{N_x}_{m=0}{g^{\alpha}_{N_x-1,m} u(x_{m})}\right] = A_{\alpha}u,
\end{equation*}
where
\begin{equation} \label{2.6}
A_{\alpha}= B_{\alpha}+ B^T_{\alpha}
~~~~{with}~~~~
    B_{\alpha}=\left [ \begin{matrix}
    g^{\alpha}_{1}  &  g^{\alpha}_{2}  &  g^{\alpha}_{3}  &    \cdots   &  g^{\alpha}_{N_x-2} &  g^{\alpha}_{N_x-1}    \\
    g^{\alpha}_{0}  &  g^{\alpha}_{1}  &  g^{\alpha}_{2}  &    \cdots   &  g^{\alpha}_{N_x-3} &  g^{\alpha}_{N_x-2}     \\
    0               &  g^{\alpha}_{0}  &  g^{\alpha}_{1}  &    \cdots   &  g^{\alpha}_{N_x-4} &  g^{\alpha}_{N_x-3}      \\
    \vdots          &      \vdots      &       \vdots     &    \ddots   &      \vdots       &    \vdots               \\
    0               &      0           &       0          &    \cdots   &  g^{\alpha}_{1}   &  g^{\alpha}_{2}          \\
    0               &      0           &       0          &    \cdots   &  g^{\alpha}_{0}   &  g^{\alpha}_{1}
    \end{matrix}
    \right ]_{\left(N_x-1\right) \times \left(N_x-1\right)}.
\end{equation}
According to (\ref{2.4}) and Lemma \ref{lemma2.1}, we can rewrite (\ref{2.0016}) as
\begin{equation}  \label{2.7}
\frac{1}{1-\gamma} \sum^{n}_{k=1} \frac{u^{k}_{i}-u^{k-1}_{i}}{\sigma\tau} e^{-\sigma(n-k)\tau} (1-e^{-\sigma\tau})
 = \frac{  \kappa_{\alpha} }{ \Gamma(4-\alpha)(\Delta x)^{\alpha} } \sum^{N_x}_{m=0}{g^{\alpha}_{i,m} u^{n}_{m}} + f^{n}_{i} +r^{n}_{i}
\end{equation}
with  the local truncation error
\begin{equation}  \label{2.8}
  |r^{n}_{i}|\leq C_u\left(\tau^2+\left(\Delta x\right)^2\right),
\end{equation}
where the positive constant $C_u$ independent of $\tau$ and $h$.

Therefore  the resulting discretization of (\ref{2.0016}) is
\begin{equation}  \label{2.9}
\frac{1}{1-\gamma} \sum^{n}_{k=1} \frac{u^{k}_{i}-u^{k-1}_{i}}{\sigma\tau} e^{-\sigma(n-k)\tau} (1-e^{-\sigma\tau})
= \frac{ \kappa_\alpha }{ \Gamma(4-\alpha)(\Delta x)^{\alpha} } \sum^{N_x}_{m=0}{g^{\alpha}_{i,m} u^{n}_{m}} + f^{n}_{i},
\end{equation}
i.e.,
\begin{equation}\label{2.10}
\begin{split}
 &  \frac{1}{(1-\gamma)\sigma \tau} \left[u^{n}_{i}\left(1-e^{-\sigma \tau}\right) - \sum^{n-1}_{k=1} {u^{k}_{i}} e^{-\sigma(n-1-k)\tau}\left(1-e^{-\sigma \tau}\right)^2  \right] \\
 &= \frac{ \kappa_\alpha }{ \Gamma(4-\alpha)(\Delta x)^{\alpha} } \sum^{N_x}_{m=0}{g^{\alpha}_{i,m} u^{n}_{m}} + f^{n}_{i} + \frac{1}{(1-\gamma)\sigma \tau} u^{0}_{i}(1-e^{-\sigma \tau})e^{-\sigma (n-1) \tau},
 \end{split}
\end{equation}
which is equivalent to
\begin{equation}\label{2.11}
 u^{n}_{i}    -  \kappa_{\Delta x,\tau}\sum^{N_x}_{m=0}{g^{\alpha}_{i,m} u^{n}_{m}}
  =  \sum^{n-1}_{k=1} {u^{k}_{i}} e^{-\sigma(n-1-k)\tau}\left(1-e^{-\sigma \tau}\right) + u^{0}_{i}e^{-\sigma (n-1) \tau}+\frac{(1-\gamma)\sigma \tau}{1-e^{-\sigma \tau}}f^{n}_{i}
\end{equation}
with $\kappa_{\Delta x,\tau}=\frac{(1-\gamma)\sigma \tau}{\left(1-e^{-\sigma \tau}\right)}  \frac{ \kappa_\alpha }{ \Gamma(4-\alpha)(\Delta x)^{\alpha} }$.
\subsection{Derivation of numerical schemes for 2D}
In the same way, we can  rewrite (\ref{1.1}) as
\begin{equation} \label{2.13}
\begin{split}
&\frac{1}{1-\gamma} \sum^{n}_{k=1} \frac{u^{k}_{i,j}-u^{k-1}_{i,j}}{\sigma\tau} e^{-\sigma(n-k)\tau} (1-e^{-\sigma\tau}) \\
& = \frac{  \kappa_{\alpha} }{ \Gamma(4-\alpha)(\Delta x)^{\alpha} } \sum^{N_x}_{m=0}{g^{\alpha}_{i,m} u^{n}_{m,j}}
  + \frac{  \kappa_{\beta} }{ \Gamma(4-\beta)(\Delta y)^{\beta} } \sum^{N_y}_{m=0}{g^{\beta}_{j,m} u^{n}_{i,m}}+ f^{n}_{i,j} +r^{n}_{i,j}
\end{split}
\end{equation}
with  the local truncation error
\begin{equation}  \label{2.14}
  |r^{n}_{i,j}|\leq C_u(\tau^2+(\Delta x)^2 + (\Delta y)^2).
\end{equation}
Therefore  the resulting discretization of (\ref{1.1}) is
\begin{equation} \label{2.15}
\begin{split}
&\frac{1}{1-\gamma} \sum^{n}_{k=1} \frac{u^{k}_{i,j}-u^{k-1}_{i,j}}{\sigma\tau} e^{-\sigma(n-k)\tau} (1-e^{-\sigma\tau}) \\
& = \frac{  \kappa_{\alpha} }{ \Gamma(4-\alpha)(\Delta x)^{\alpha} } \sum^{N_x}_{m=0}{g^{\alpha}_{i,m} u^{n}_{m,j}}
  + \frac{  \kappa_{\beta} }{ \Gamma(4-\beta)(\Delta y)^{\beta} } \sum^{N_y}_{m=0}{g^{\beta}_{j,m} u^{n}_{i,m}}+ f^{n}_{i,j},
\end{split}
\end{equation}
i.e.,
\begin{equation}\label{2.16}
\begin{split}
 & u^{n}_{i,j}-\kappa^{\alpha}_{\Delta x,\tau}\sum^{N_x}_{m=0}{g^{\alpha}_{i,m}u^{n}_{m,j}}-\kappa^{\alpha}_{\Delta y,\tau}\sum^{N_y}_{m=0}{g^{\alpha}_{j,m} u^{n}_{i,m}}\\
 &=  \sum^{n-1}_{k=1} {u^{k}_{i,j}} e^{-\sigma(n-1-k)\tau}\left(1-e^{-\sigma \tau}\right) + u^{0}_{i,j}e^{-\sigma (n-1) \tau}+\frac{(1-\gamma)\sigma \tau}{1-e^{-\sigma \tau}}f^{n}_{i,j}
\end{split}
\end{equation}
with $\kappa^{\alpha}_{\Delta x,\tau}=\frac{(1-\gamma)\sigma \tau}{\left(1-e^{-\sigma \tau}\right)}  \frac{ \kappa_\alpha }{ \Gamma(4-\alpha)(\Delta x)^{\alpha} }$ and $\kappa^{\beta}_{\Delta y,\tau}=\frac{(1-\gamma)\sigma \tau}{\left(1-e^{-\sigma \tau}\right)}  \frac{ \kappa_\beta }{ \Gamma(4-\beta)(\Delta y)^{\beta} }$.

\section{Stability and convergence}
In this section, we theoretically prove that the above numerical schemes are unconditionally stable with the nonzero initial conditions.
First, we denote $u^{n}=\{u^{n}_{i}|0\leq i \leq N_x,n \geq 0\}$ and $v^{n}=\{v^{n}_{i}|0\leq i \leq N_x,n \geq 0\}$, which are grid functions. And we introduce the discrete inner products as following
$$(u^{n},v^{n}) = \Delta x \sum^{N_x-1}_{i=0}{u^{n}_{i}v^{n}_{i}},~~~||u^{n}||=(u^{n},u^{n})^{1/2}.$$
\subsection{A few technical lemmas}
\begin{lemma}\cite[p.\,27]{Quarteroni:07} \label{lemma3.1}
 A matrix $A \in R $ is positive definite in $R$ if $(Ax,x)>0, {\forall}x \in R^{n}, x \neq 0$. A real matrix $A$ of order $n$ is positive definite iff its symmetric part $H=\frac{A+A^{T}}{2}$ is positive definite.
\end{lemma}
\begin{lemma} \cite[p.\,29]{Quarteroni:07}  \label{lemma3.2}
  A matrix $A \in R^{n \times n}$ is called  diagonally dominant by rows if $|a_{ii}| \geq \sum^{n}_{j=1,j \neq i}{|a_{ij}|} $ with $i = 1,2,\cdots n$. A strictly diagonally dominant matrix that is symmetric with positive diagonal entries is also positive definite.
\end{lemma}
\begin{lemma} [\cite{Chen:14}] \label{lemma3.3}
The coefficients $\widetilde g^{\alpha}_{i,m}$, $\alpha \in (1,2)$ defined in (\ref{2.5}) satisfy
\begin{equation}
\begin{split}
      & (1)~ g^{\alpha}_{i,i} < 0 , ~~~ g^{\alpha}_{i,m} > 0 ~~(m \neq i); \\
      & (2)~\sum^{N_x}_{m=0}{ g^{\alpha}_{i,m}} < 0 ~~~and~~~ - g^{\alpha}_{i,i} > \sum^{N_x}_{ m=0 ,m \neq i  }{ g^{\alpha}_{i,m}};\\
      & (3)~g_0^\alpha>0,~ g_1^\alpha<0,~ g_0^\alpha+g_2^\alpha>0, ~g_k^\alpha>0~~\forall k\geq 3.
\end{split}
\end{equation}
\end{lemma}

\begin{lemma}  \label{lemma3.4}
Let $1<\alpha <2$ and $ g^{\alpha}_{m}$ given by (\ref{2.3}). Then
\begin{equation} \label{3.1}
    \sum^{i+1}_{m=0}{ g^{\alpha}_{m}} \leq \frac{1}{i^{\alpha} \Gamma(1-\alpha)} <0. 
\end{equation}
\end{lemma}
\begin{proof}
Using (2.13) and (2.4) with $u(x,t)=1$ of {\cite{Chen:014}}, we obtain
\begin{equation*}
\begin{split}
&\sum^{i+1}_{m=0} {g_{m}}\\
&= \frac{\Gamma(4- \alpha)(\Delta x)^\alpha}{(\Delta x)^{2}\Gamma(2- \alpha)}
 \left[  \int_0^{x_{i-1}} \!\! (x_{i-1}-\xi)^{1-\alpha} d\xi \!\!-\!\! 2 \int_0^{x_{i}} \!\! (x_{i}-\xi)^{1-\alpha} d\xi \!\!+\!\! \int_0^{x_{i+1}} \!\! (x_{i+1}\!\!-\!\!\xi)^{1-\alpha} d\xi \right]\\
    & = (3- \alpha) i^{2- \alpha} \left[ \left(1- \frac{1}{i}\right)^{2- \alpha} -2 + \left(1+ \frac{1}{i}\right)^{2- \alpha} \right] \\
    & \leq  \frac{(3-\alpha)(2-\alpha)(1-\alpha)}{i^{\alpha} }\leq \frac{1}{i^{\alpha} \Gamma(1-\alpha)} <0.
  \end{split}
  \end{equation*}
The proof is completed.
\end{proof}

\begin{lemma} \label{lemma3.5}
Let $1<\alpha<2$ and $A_\alpha $ be given in (\ref{2.6}). Then
\begin{equation*}
-  \frac{1}{(\Delta x)^\alpha}(A_\alpha v,v)\geq -\frac{2}{\Gamma(1-\alpha)(x_R)^\alpha}||v||^2>0 ~~{\rm with}~v\in {\mathbb{R}}^{M-1},~\Omega=(0,x_R).
\end{equation*}
\end{lemma}
\begin{proof}
Let the vector $v=(v_1,v_2,\ldots,v_{N_x-1})^T$ with $v_0=v_{N_x}=0$.
From (\ref{2.6}) and Lemma  \ref{lemma3.3}, there exists
\begin{equation*}
\begin{split}
(B_\alpha v,v)
&=\Delta x\sum_{i=1}^{N_x-1} \left(\sum_{k=0}^{N_x-i}g^\alpha_kv_{i+k-1}  \right)v_i
              =\Delta x\sum_{k=0}^{N_x-1}g^\alpha_k \left(\sum_{i=1}^{N_x-k}v_{i+k-1} v_i \right)\\
&=g^\alpha_1 \Delta x\sum_{i=1}^{N_x-1}v_i^2 +\left(g^\alpha_0+g^\alpha_2\right) \Delta x\sum_{i=1}^{N_x-2}v_iv_{i+1}+\Delta x\sum_{k=3}^{N_x-1}g^\alpha_k \left(\sum_{i=1}^{N_x-k}v_{i+k-1} v_i \right)\\
&\leq g^\alpha_1 ||v||^2 +\left(g^\alpha_0+g^\alpha_2\right) \Delta x\sum_{i=1}^{N_x-2}\frac{v_i^2+v_{i+1}^2}{2}+\Delta x\sum_{k=3}^{N_x-1}g^\alpha_k \left(\sum_{i=1}^{N_x-k}\frac{v_i^2+v_{i+k-1}^2}{2} \right)\\
&\leq  \left(\sum_{k=0}^{M-1}g^\alpha_k \right)||v||^2\leq  \left(\sum_{k=0}^{N_x+1}g^\alpha_k \right)||v||^2.
\end{split}
\end{equation*}
Since $(B^T_\alpha v,v)=(B_\alpha v,v)$ and $A_\alpha=B^T_\alpha+B_\alpha$, we have
\begin{equation*}
  (A_\alpha v,v)\leq  2\left(\sum_{k=0}^{N_x+1}g^\alpha_k \right)||v||^2<0.
\end{equation*}
Using the above inequality  and Lemma \ref{lemma3.4}, we obtain
\begin{equation*}
\begin{split}
-  \frac{1}{(\Delta x)^\alpha}(A_\alpha v,v)
&\geq -\frac{2}{(\Delta x)^\alpha} \left(\sum_{k=0}^{N_x+1}g^\alpha_k \right)||v||^2
\geq \frac{2}{(\Delta x)^\alpha}   \frac{-1}{(N_x)^\alpha\Gamma(1-\alpha)}||v||^2\\
&\geq -\frac{2}{(x_R)^\alpha\Gamma(1-\alpha)}||v||^2>0 ~~{\rm with}~~\Omega=(0,x_R),~~v \in {\mathbb{R}}^{N_x-1}.
\end{split}
\end{equation*}
The proof is completed.
\end{proof}

\begin{lemma}  \label{lemma3.6}
Let $0<\gamma<1$ and $\sigma=\frac{\gamma}{1-\gamma}$. Then for any vector  $V_{i}=(v^{1}_{i},v^{2}_{i},\cdots,v^{N}_{i}) \in R^{N}$, we have
\begin{equation*}
     \sum ^{N}_{n=1}{ \left( v^{n}_{i}(1-e^{-\sigma \tau})
     -  \sum^{n-1}_{k=1} {v^{k}_{i}}  e^{-\sigma(n-1-k)\tau}\left(1-e^{-\sigma \tau}\right)^2 \right) v^{n}_{i} } \geq 0.
\end{equation*}
\end{lemma}
\begin{proof}
By the mathematical induction method, we have
\begin{equation*} \label{3.3}
\begin{split}
\sum ^{N}_{n=1}{ \left( v^{n}_{i}(1-e^{-\sigma \tau})
- \sum^{n-1}_{k=1} {v^{k}_{i}}  e^{-\sigma(n-1-k)\tau}\left(1-e^{-\sigma \tau}\right)^2 \right) v^{n}_{i} }
 = (AV_{i},V_{i}),
\end{split}
\end{equation*}
where
\begin{equation*}
    A=\left [ \begin{matrix}
    b            &         0       &         0          &     \cdots   &         0         &       0    \\
    a_{1}        &         b       &         0          &     \cdots   &         0         &       0     \\
    a_{2}        &       a_{1}     &         b          &     \cdots   &         0         &       0      \\
    \vdots       &      \vdots     &       \vdots       &     \ddots   &      \vdots       &    \vdots     \\
    a_{N-2}      &      a_{N-3}    &      a_{N-4}       &     \cdots   &         b         &       0        \\
    a_{N-1}      &      a_{N-2}    &      a_{N-3}       &     \cdots   &       a_{1}       &       b
\end{matrix}
    \right ]_{N \times N}
\end{equation*}
with $b=1-e^{-\sigma \tau}$ and $ a_{l} = - e^{-\sigma(l-1)\tau}\left(1-e^{-\sigma \tau}\right)^2 < 0, l=1,2,\cdots,N-1$.
We next prove the matrix $A$ is positive definite.  Since
\begin{equation*}
    H=\frac{A+A^T}{2}=\left [ \begin{matrix}
     b                 & \frac{a_{1} }{2}   &  \frac{a_{2} }{2}    &   \cdots   & \frac{a_{N-2} }{2} &  \frac{a_{N-1} }{2}   \\
    \frac{a_{1} }{2}   &  b                 &  \frac{a_{1} }{2}    &   \cdots   & \frac{a_{N-3} }{2} &  \frac{a_{N-2} }{2}    \\
    \frac{a_{2} }{2}   & \frac{a_{1} }{2}   &   b                  &   \cdots   & \frac{a_{N-4} }{2} &  \frac{a_{N-3} }{2}     \\
    \vdots             & \vdots             &  \vdots              &   \ddots   & \vdots             &  \vdots                  \\
    \frac{a_{N-2} }{2} & \frac{a_{N-3} }{2} &  \frac{a_{N-4} }{2}  &   \cdots   &  b                 &  \frac{a_{1} }{2}         \\
    \frac{a_{N-1} }{2} & \frac{a_{N-2} }{2} &  \frac{a_{N-3} }{2}  &   \cdots   & \frac{a_{1} }{2}   &   b
    \end{matrix}
    \right ]_{N \times N}
\end{equation*}
and
\begin{equation*}
\begin{split}
       \sum^{N}_{j=1,j\neq i}{|h_{i,j}|}
       & \leq 2 \sum^{N-1}_{l=1}{\Big|\frac{a_{l}}{2}\Big|} = - \sum^{N-1}_{l=1}{a_{l}}
         =  \sum^{N-1}_{l=1}e^{-\sigma(l-1)\tau}\left(1-e^{-\sigma \tau}\right)^2 \\
       & =  \left(1-e^{-\sigma \tau}\right)\left(1-  e^{- \sigma (N-1) \tau} \right) <  1-e^{-\sigma \tau} = b,
\end{split}
\end{equation*}
it yields  the matrix $H$ is strictly diagonally dominant.
Form Lemmas \ref{lemma3.1} and \ref{lemma3.2}, we know that the matrix $A$ is positive definite. The proof is completed.
\end{proof}
For simplifying the proof of the stability  and convergence,  we first provide the following  a priori estimate.
\begin{lemma}  \label{lemma3.7}
Suppose ${v^{n}_{i}}$ is the solution of the difference scheme (\ref{2.10}), i.e.,
\begin{equation} \label{3.3}
\begin{split}
    &  \frac{1}{(1-\gamma)\sigma \tau} \left[ v^{n}_{i}(1-e^{-\sigma \tau}) -  \sum^{n-1}_{k=1} v^{k}_{i}  e^{-\sigma(n-1-k)\tau}\left(1-e^{-\sigma \tau}\right)^2\right] \\
    &\quad = \frac{ \kappa_\alpha }{ \Gamma(4-\alpha)(\Delta x)^{\alpha} } \sum^{N_x}_{m=0}{g^{\alpha}_{i,m} v^{n}_{m}} + f^{n}_{i} + \frac{1}{(1-\gamma)\sigma \tau} v^{0}_{i}(1-e^{-\sigma \tau})e^{-\sigma (n-1) \tau}, \\
    & v^{0}_{i} = \phi _{i},~~~1 \leq i \leq N_x-1,\\
    & v^{n}_{0} = v^{n}_{M} = 0, ~~~ 0 \leq n \leq N.
\end{split}
\end{equation}
Then for any positive integer $N$ with $N \tau \leq T$, we have
\begin{equation*}
\begin{split}
&\tau \sum^{N}_{n=1}{||v^{n}||^2 }\\
& \leq  \frac{ \left(\Gamma(1-\alpha) \Gamma(4-\alpha )  (x_R)^{ \alpha}\right)^2 }{ 2  \kappa_\alpha^2} \cdot \tau\sum^{N}_{n=1} ||f^{n}||^2
       + \frac{     \left(\Gamma(1-\alpha) \Gamma(4-\alpha)  (x_R)^{  \alpha}\right)^2 T }{ (1-\gamma)^2 \kappa_\alpha^2 } ||v^{0}||^2.
\end{split}
\end{equation*}
\end{lemma}
\begin{proof}
Multiplying (\ref{3.3}) by $(\Delta x)v^{n+1}_{i}$ and summing up for $i$ from $1$ to $N_x-1$, we get
\begin{equation*}
\begin{split}
&\frac{1}{(1-\gamma)\sigma \tau}\sum^{N_x-1}_{i=1} \left[ v^{n}_{i}(1-e^{-\sigma \tau}) -  \sum^{n-1}_{k=1} v^{k}_{i}e^{-\sigma(n-1-k)\tau}\left(1-e^{-\sigma \tau}\right)^2 \right] hv^{n}_{i}\\
&= \sum^{N_x-1}_{i=1}\left[\frac{ \kappa_\alpha }{ \Gamma(4-\alpha)(\Delta x)^{\alpha} } \sum^{N_x}_{m=0}{g^{\alpha}_{i,m} v^{n}_{m}} + f^{n}_{i}
   + \frac{1}{(1-\gamma)\sigma \tau} v^{0}_{i}(1-e^{-\sigma \tau})e^{-\sigma (n-1) \tau} \right] (\Delta x)v^{n}_{i} \\
&=\frac{ \kappa_\alpha }{ \Gamma(4-\alpha)(\Delta x)^{\alpha} } (A_{\alpha} v^{n} , v^{n} ) + (f^{n} , v^{n} ) + \frac{1}{(1-\gamma)\sigma \tau} (1-e^{-\sigma \tau})e^{-\sigma (n-1) \tau} (v^{0} , v^{n} ).
\end{split}
\end{equation*}
Multiplying the above equation  by $\tau $ and summing up for $n$ from $1$ to $N$, we obtain
\begin{equation*}
\begin{split}
& \frac{\tau}{(1-\gamma)\sigma\tau} \sum^{N_x-1}_{i=1} \sum^{N}_{n=1} \left[ v^{n}_{i}(1-e^{-\sigma \tau})
- \sum^{n-1}_{k=1} v^{k}_{i} e^{-\sigma(n-1-k)\tau}\left(1-e^{-\sigma \tau}\right)^2 \right] (\Delta x)v^{n}_{i}\\
& = \tau \sum^{N}_{n=1}{\frac{ \kappa_\alpha }{ \Gamma(4-\alpha)(\Delta x)^{\alpha} } (A_{\alpha} v^{n} , v^{n} )} + \tau \sum^{N}_{n=1} (f^{n} , v^{n} )
+  \frac{\tau  (1-e^{-\sigma \tau})}{(1-\gamma)\sigma \tau} \sum^{N}_{n=1}e^{-\sigma (n-1) \tau} (v^{0} , v^{n} ).
\end{split}
\end{equation*}
From  Lemma \ref{lemma3.5} and Lemma \ref{lemma3.6}, we have
\begin{equation*}
\begin{split}
      &  -\frac{2}{ \Gamma(1-\alpha)} \frac{ \kappa_\alpha}{ \Gamma(4-\alpha) (x_R)^{\alpha} } \tau \sum^{N}_{n=1}{||v^{n}||^{2} } \leq  - \tau \sum^{N}_{n=1}{\frac{ \kappa_\alpha }{ \Gamma(4-\alpha)(\Delta x)^{\alpha} } (A_{\alpha} v^{n} , v^{n} )} \\
      & \leq \tau \sum^{N}_{n=1} (f^{n} , v^{n} ) + \frac{1}{(1-\gamma)\sigma \tau} (1-e^{-\sigma \tau}) \tau \sum^{N}_{n=1}e^{-\sigma (n-1) \tau} (v^{0} , v^{n} ) \\
      & \leq \tau \sum^{N}_{n=1} ||f^{n}|| \cdot ||v^{n}|| + \frac{1}{(1-\gamma)} \tau \sum^{N}_{n=1}  ||v^{0}|| \cdot ||v^{n}|| \\
      & \leq \tau \sum^{N}_{n=1}{(\epsilon ||v^{n}||^2 + \frac{||f^{n}||^2}{4 \epsilon} ) } + \frac{1}{(1-\gamma)} \tau \sum^{N}_{n=1} \left(\eta ||v^{n}||^2 + \frac{||v^{0}||^2}{4 \eta} \right),
\end{split}
\end{equation*}
where $\epsilon, \eta >0 $ and $(f^{n} , v^{n}) \leq ||f^{n}|| \cdot ||v^{n}||$.

Taking $\epsilon = - \frac{1}{  \Gamma(1-\alpha)} \frac{ \kappa_\alpha }{ \Gamma(4-\alpha) (x_R)^{\alpha} } $ and $ \eta = - \frac{ 1 - \gamma}{ 2 \Gamma(1-\alpha)} \frac{ \kappa_\alpha }{ \Gamma(4-\alpha) (x_R)^{\alpha} } $,
and using the above inequality, there exists
\begin{equation*}
\begin{split}
        & - \frac{1}{ 2 \Gamma(1-\alpha)} \frac{ \kappa_\alpha}{ \Gamma(4-\alpha) (x_R)^{\alpha} } \tau  \sum^{N}_{n=1}{||v^{n}||^2 }  \leq \frac{1}{4 \epsilon} \tau \sum^{N}_{n=1}{ ||f^{n}||^2 } + \frac{1}{(1-\gamma)4 \eta } \tau \sum^{N}_{n=1}  ||v^{0}||^2 ,
\end{split}
\end{equation*}
and
\begin{equation*}
\begin{split}
\tau \sum^{N}_{n=1}{||v^{n}||^2 }
& \leq \frac{1-\gamma}{4 \epsilon \eta } \cdot \tau\sum^{N}_{n=1}{ ||f^{n}||^2 } + \frac{T}{4 \eta^2 }  ||v^{0}||^2 \\
& = \frac{ \left(\Gamma(1-\alpha) \Gamma(4-\alpha )  (x_R)^{ \alpha}\right)^2 }{ 2  \kappa_\alpha^2}  \tau\sum^{N}_{n=1} ||f^{n}||^2
      \! +\! \frac{     \left(\Gamma(1-\alpha) \Gamma(4-\alpha)  (x_R)^{  \alpha}\right)^2 T }{ (1-\gamma)^2 \kappa_\alpha^2 } ||v^{0}||^2.
\end{split}
\end{equation*}
The proof is completed.
\end{proof}

\subsection{Convergence and stability for 1D  }
In this subsection, we prove that the scheme (\ref{2.10}) is unconditionally stable and convergence by two methods, i.e.,  a prior estimate and mathematical induction
which correspond to  the discrete $L^2$-norm and $L_\infty$ norm.
\begin{theorem} \label{theorem3.1}
The difference scheme (\ref{2.10}) is unconditionally stable.
\begin{proof}
  From Lemma \ref{lemma3.7}, the desired results is obtained.
\end{proof}
\end{theorem}

\begin{theorem} \label{theorem3.2}
Let $u^n_i$ be the approximate solution of $u(x_i,t_n)$ computed by the difference scheme (\ref{2.10}). Let $\varepsilon ^n_i = u(x_i,t_n)- u^n_i $. Then
\begin{equation*}
      \tau \sum^{N}_{n=1}||\varepsilon^{n}|| \leq
       \frac{ |\Gamma(1-\alpha)| \Gamma(4-\alpha )  (x_R)^{ \alpha+1/2} T}{ \sqrt{2}  \kappa_\alpha}  C_{u} \cdot (\tau^{2}+(\Delta x)^{2}),
\end{equation*}
where $C_{u}$ is defined by (\ref{2.8}) and $(x_{i},t_{n}) \in (0,x_R) \times (0,T]$ with $N\tau \leq T$.
\end{theorem}
\begin{proof}
Let $u(x_{i},t_{n})$ be the exact solution of (\ref{2.0016}) at the mesh point $(x_{i},t_{n})$, and $\varepsilon ^n_i = u(x_i,t_n)- u^n_i $.
Subtracting (\ref{2.7}) from (\ref{2.10}) with $\varepsilon ^0_i = 0$, we obtain
\begin{equation*}
\begin{split}
       &  \frac{1}{(1-\gamma)\sigma \tau} \left[ \varepsilon^{n}_{i}(1-e^{-\sigma \tau}) -  \sum^{n-1}_{k=1} \varepsilon^{k}_{i}  e^{-\sigma(n-1-k)\tau}\left(1-e^{-\sigma \tau}\right)^2   \right]\\
       &= \frac{ \kappa_\alpha }{ \Gamma(4-\alpha)(\Delta x)^{\alpha} } \sum^{N_x}_{m=0}{g^{\alpha}_{i,m} \varepsilon^{n}_{m}} + r^{n}_{i} .
\end{split}
\end{equation*}
From Lemma \ref{lemma3.7} and (\ref{2.8}), it holds
\begin{equation*}
\begin{split}
\tau \sum^{N}_{n=1}{||\varepsilon^{n}||^2 }
& \leq  \frac{ \left(\Gamma(1-\alpha)\Gamma(4-\alpha )(x_R)^{ \alpha}\right)^2 }{ 2 \kappa_\alpha^2} \cdot \tau\sum^{N}_{n=1} ||r^{n}||^2\\
& \leq  \frac{ \left(\Gamma(1-\alpha)\Gamma(4-\alpha )(x_R)^{ \alpha}\right)^2 }{ 2 \kappa_\alpha^2} x_R T C^{2}_{u}\cdot (\tau^{2}+(\Delta x)^{2})^{2}.
\end{split}
\end{equation*}
Using Cauchy-Schwarz inequality for the above inequality, we have
\begin{equation*}
\begin{split}
      \left( \tau \sum^{N}_{n=1}||\varepsilon^{n}|| \right)^{2}
      & \leq \left(  \tau \sum^{N}_{n=1} {1} \right) \left( \tau \sum^{N}_{n=1}||\varepsilon^{n}||^2 \right) \\
      & \leq \frac{ \left(\Gamma(1-\alpha) \Gamma(4-\alpha )  (x_R)^{ \alpha}\right)^2 }{ 2  \kappa_\alpha^2} x_R T^2 C^{2}_{u} \cdot (\tau^{2}+(\Delta x)^{2})^{2},
\end{split}
\end{equation*}
and
\begin{equation*}
      \tau \sum^{N}_{n=1}||\varepsilon^{n}|| \leq
       \frac{ |\Gamma(1-\alpha)| \Gamma(4-\alpha )  (x_R)^{ \alpha+1/2} T}{ \sqrt{2}  \kappa_\alpha}  C_{u} \cdot (\tau^{2}+(\Delta x)^{2}).
\end{equation*}
The proof is completed.
\end{proof}

Besides the discrete $L^2$-norm, the stability and convergence can also be obtained in $L_\infty$ norm by the following mathematical induction.
\begin{theorem} \label{theorem3.3}
The difference scheme (\ref{2.11}) is unconditionally stable.
\end{theorem}
\begin{proof}
Let $ \widetilde u^{n}_{i} $ be the approximate solution of $u^{n}_{i}$, which is the exact solution of the difference scheme (\ref{2.11}).
Denoting $ \epsilon ^n_i = \widetilde u^{n}_{i}- u^n_i $, there exists
\begin{equation*}
\begin{split}
& \left( 1 -\kappa_{\Delta x,\tau}  g^{\alpha}_{i,i} \right) \epsilon^{1}_{i} -\kappa_{\Delta x,\tau} \sum ^{N_x}_{m=0,m \neq i} { g^{\alpha}_{i,m} \epsilon^{1}_{m}}
= \epsilon^{0}_{i}e^{-\sigma (n-1) \tau}, ~~~~ n=1, \\
&\left( 1 -\kappa_{\Delta x,\tau}  g^{\alpha}_{i,i} \right)\epsilon^{n}_{i} \!-\!\kappa_{\Delta x,\tau} \!\!\sum ^{N_x}_{m=0,m \neq i} { g^{\alpha}_{i,m} \epsilon^{n}_{m}}
       =  \sum^{n-1}_{k=1} {\epsilon^{k}_{i}} e^{-\sigma(n-1-k)\tau}\left(1-e^{-\sigma \tau}\right) + \epsilon^{0}_{i}e^{-\sigma (n-1) \tau}, ~ n>1.
\end{split}
\end{equation*}

We next prove the scheme is unconditionally stable by the mathematical induction.

Let $E^{n} = [\epsilon^{n}_{0},\epsilon^{n}_{1},\cdots,\epsilon^{n}_{N_x}]$ and $ |\epsilon^{n}_{i_0}|:=||E^{n}||_{\infty} = \max \limits_{0 \leq i \leq N_x} {|\epsilon^{n}_{i}|}$.
From  Lemma \ref{lemma3.3}, we obtain
\begin{equation*}
\begin{split}
||E^{1}||_{\infty}
& = |\epsilon^{1}_{i_0}| \leq |\epsilon^{1}_{i_0}| -\kappa_{\Delta x,\tau} \sum ^{N_x}_{m=0} { g^{\alpha}_{i_0,m} |\epsilon^{1}_{i_0}}|
 = |\epsilon^{1}_{i_0}| -\kappa_{\Delta x,\tau}{ g^{\alpha}_{i_0,i_0} |\epsilon^{1}_{i_0}}|
         -\kappa_{\Delta x,\tau} \sum ^{N_x}_{m=0,m \neq i_0} { g^{\alpha}_{i_0,m} |\epsilon^{1}_{i_0}}| \\
& \leq \left(1 -\kappa_{\Delta x,\tau} g^{\alpha}_{i_0,i_0}   \right) |\epsilon^{1}_{i_0}|
         -\kappa_{\Delta x,\tau} \sum ^{N_x}_{m=0,m \neq i_0} { g^{\alpha}_{i_0,m} |\epsilon^{1}_{m}}|\\
& \leq \left|\left(1 -\kappa_{\Delta x,\tau} g^{\alpha}_{i_0,i_0}   \right) \epsilon^{1}_{i_0}
         -\kappa_{\Delta x,\tau} \sum ^{N_x}_{m=0,m \neq i_0} { g^{\alpha}_{i_0,m} \epsilon^{1}_{m}}\right|
= |\epsilon^{0}_{i_0}|e^{-\sigma (n-1) \tau}\leq ||E^{0}||_{\infty}.
\end{split}
\end{equation*}
Assuming $||E^s||_{\infty} \leq ||E^{0}||_{\infty}$, $s = 1, 2,3,\cdots,n-1 $, we have
\begin{equation*}
\begin{split}
||E^{n}||_{\infty}
& = |\epsilon^{n}_{i_0}| \leq |\epsilon^{n}_{i_0}| -\kappa_{\Delta x,\tau} \sum ^{N_x}_{m=0} { g^{\alpha}_{i_0,m} |\epsilon^{n}_{i_0}}|
 = |\epsilon^{n}_{i_0}| -\kappa_{\Delta x,\tau}{ g^{\alpha}_{i_0,i_0} |\epsilon^{n}_{i_0}}|
         -\kappa_{\Delta x,\tau} \sum ^{N_x}_{m=0,m \neq i_0} { g^{\alpha}_{i_0,m} |\epsilon^{n}_{i_0}}| \\
& \leq \left(1 -\kappa_{\Delta x,\tau} g^{\alpha}_{i_0,i_0}   \right) |\epsilon^{n}_{i_0}|
         -\kappa_{\Delta x,\tau} \sum ^{N_x}_{m=0,m \neq i_0} { g^{\alpha}_{i_0,m} |\epsilon^{n}_{m}}|\\
& \leq \left|\left(1 -\kappa_{\Delta x,\tau} g^{\alpha}_{i_0,i_0}   \right) \epsilon^{n}_{i_0}
         -\kappa_{\Delta x,\tau} \sum ^{N_x}_{m=0,m \neq i_0} { g^{\alpha}_{i_0,m} \epsilon^{n}_{m}}\right|\\
& = \left| \sum^{n-1}_{k=1} {\epsilon^{k}_{i_0}} e^{-\sigma(n-1-k)\tau}\left(1-e^{-\sigma \tau}\right) + \epsilon^{0}_{i_0}e^{-\sigma (n-1) \tau}\right|\\
&\leq  \left| \sum^{n-1}_{k=1}  e^{-\sigma(n-1-k)\tau}\left(1-e^{-\sigma \tau}\right) +e^{-\sigma (n-1) \tau}\right| ||E^{0}||_{\infty}= ||E^{0}||_{\infty}.
\end{split}
\end{equation*}
The proof is completed.
\end{proof}

\begin{theorem}
Let $u^n_i$ be the approximate solution of $u(x_i,t_n)$ computed by the difference scheme (\ref{2.11}). Let $\varepsilon ^n_i = u(x_i,t_n)- u^n_i $. Then
 \begin{equation*}
    || \varepsilon^{n} || _{\infty} \leq ( 1 - \gamma ) \left(\sigma T+e^{ \sigma \tau}\right)  C_u(\tau^2+(\Delta x)^2),
 \end{equation*}
where $C_{u}$ is defined by (\ref{2.8}) and $(x_{i},t_{n}) \in (0,b) \times (0,T]$ with $N\tau \leq T$.
\end{theorem}
\begin{proof}
Let $u(x_{i},t_{n})$ be the exact solution of (\ref{2.0016}) at the mesh point $(x_{i},t_{n})$. Defined $\varepsilon ^{n}_{i} = u(x_i,t_n)- u^n_i $
and $\varepsilon^n = [\varepsilon ^{n}_{0},\varepsilon ^{n}_{1},\cdots, \varepsilon ^{n}_{N_x}]$.
Subtracting (\ref{2.7}) from (\ref{2.11}) with $\varepsilon ^{0}_{i} = 0$, we obtain
  \begin{equation*}
\begin{split}
& \left( 1 -\kappa_{\Delta x,\tau}  g^{\alpha}_{i,i} \right) \epsilon^{1}_{i} -\kappa_{\Delta x,\tau} \sum ^{N_x}_{m=0,m \neq i} { g^{\alpha}_{i,m} \epsilon^{1}_{m}}
= \frac{(1-\gamma)\sigma \tau}{1-e^{-\sigma \tau}}r^{1}_{i}, ~~~~ n=1, \\
&\left( 1 -\kappa_{\Delta x,\tau}  g^{\alpha}_{i,i} \right)\epsilon^{n}_{i} \!-\!\kappa_{\Delta x,\tau} \!\!\sum ^{N_x}_{m=0,m \neq i} { g^{\alpha}_{i,m} \epsilon^{n}_{m}}
       =  \sum^{n-1}_{k=1} {\epsilon^{k}_{i}} e^{-\sigma(n-1-k)\tau}\left(1-e^{-\sigma \tau}\right) \!+\!\frac{(1-\gamma)\sigma \tau}{1-e^{-\sigma \tau}}r^{n}_{i}, ~ n>1.
\end{split}
\end{equation*}

We next prove the desired results  by the mathematical induction.

Let  $ |\epsilon^{n}_{i_0}|:=||\epsilon^{n}||_{\infty} = \max \limits_{0 \leq i \leq N_x} {|\epsilon^{n}_{i}|}$
and $r_{\max} = \max\limits_{0 \leq i \leq N_x,0 \leq n \leq N}{|r^{n}_{i}|}$. Using  Lemma \ref{lemma3.3}, we have
\begin{equation*}
\begin{split}
||\epsilon^{1}||_{\infty}
& = |\epsilon^{1}_{i_0}| \leq |\epsilon^{1}_{i_0}| -\kappa_{\Delta x,\tau} \sum ^{N_x}_{m=0} { g^{\alpha}_{i_0,m} |\epsilon^{1}_{i_0}}|
 = |\epsilon^{1}_{i_0}| -\kappa_{\Delta x,\tau}{ g^{\alpha}_{i_0,i_0} |\epsilon^{1}_{i_0}}|
         -\kappa_{\Delta x,\tau} \sum ^{N_x}_{m=0,m \neq i_0} { g^{\alpha}_{i_0,m} |\epsilon^{1}_{i_0}}| \\
& \leq \left(1 -\kappa_{\Delta x,\tau} g^{\alpha}_{i_0,i_0}   \right) |\epsilon^{1}_{i_0}|
         -\kappa_{\Delta x,\tau} \sum ^{N_x}_{m=0,m \neq i_0} { g^{\alpha}_{i_0,m} |\epsilon^{1}_{m}}|\\
& \leq \left|\left(1 -\kappa_{\Delta x,\tau} g^{\alpha}_{i_0,i_0}   \right) \epsilon^{1}_{i_0}
         -\kappa_{\Delta x,\tau} \sum ^{N_x}_{m=0,m \neq i_0} { g^{\alpha}_{i_0,m} \epsilon^{1}_{m}}\right|
= \left|\frac{(1-\gamma)\sigma \tau}{1-e^{-\sigma \tau}}r^{1}_{i_0}\right|\leq \frac{(1-\gamma)\sigma \tau}{1-e^{-\sigma \tau}}r_{\max}.
\end{split}
\end{equation*}
Supposing  $  || e^{n} || _{\infty} = |\varepsilon ^{n}_{i_0}| = \max\limits_{0 \leq i \leq N_x}{|\varepsilon ^{n}_{i}|}  $ and
$$ || e^{s } || _{\infty} \leq \frac{ (s-1) \left(1-e^{- \sigma \tau} \right)+1}{ 1 - e^{- \sigma \tau}}  ( 1 - \gamma ) \sigma \tau  r_{\max},~~ s = 1,2,3,\cdots,n-1,$$
we have
\begin{equation*}
\begin{split}
||\epsilon^{n}||_{\infty}
& = |\epsilon^{n}_{i_0}| \leq |\epsilon^{n}_{i_0}| -\kappa_{\Delta x,\tau} \sum ^{N_x}_{m=0} { g^{\alpha}_{i_0,m} |\epsilon^{n}_{i_0}}|
 = |\epsilon^{n}_{i_0}| -\kappa_{\Delta x,\tau}{ g^{\alpha}_{i_0,i_0} |\epsilon^{n}_{i_0}}|
         -\kappa_{\Delta x,\tau} \sum ^{N_x}_{m=0,m \neq i_0} { g^{\alpha}_{i_0,m} |\epsilon^{n}_{i_0}}| \\
& \leq \left(1 -\kappa_{\Delta x,\tau} g^{\alpha}_{i_0,i_0}   \right) |\epsilon^{n}_{i_0}|
         -\kappa_{\Delta x,\tau} \sum ^{N_x}_{m=0,m \neq i_0} { g^{\alpha}_{i_0,m} |\epsilon^{n}_{m}}|\\
& \leq \left|\left(1 -\kappa_{\Delta x,\tau} g^{\alpha}_{i_0,i_0}   \right) \epsilon^{n}_{i_0}
         -\kappa_{\Delta x,\tau} \sum ^{N_x}_{m=0,m \neq i_0} { g^{\alpha}_{i_0,m} \epsilon^{n}_{m}}\right|\\
& = \left| \sum^{n-1}_{k=1} {\epsilon^{k}_{i_0}} e^{-\sigma(n-1-k)\tau}\left(1-e^{-\sigma \tau}\right) +\frac{(1-\gamma)\sigma \tau}{1-e^{-\sigma \tau}}r^{n}_{i_0}\right|
 \leq \Phi_\tau \cdot r_{\max}
\end{split}
\end{equation*}
with
\begin{equation} \label{3.004}
\begin{split}
\Phi_\tau
&= \sum^{n-1}_{k=1} \frac{ k - ( k - 1)e^{- \sigma \tau} }{ 1 - e^{- \sigma \tau}}  ( 1 - \gamma ) \sigma \tau
\cdot e^{-\sigma(n-1-k)\tau}\left(1-e^{-\sigma \tau}\right) +\frac{(1-\gamma)\sigma \tau}{1-e^{-\sigma \tau}}\\
&=\frac{ (n-1) \left(1-e^{- \sigma \tau} \right)+1}{ 1 - e^{- \sigma \tau}}  ( 1 - \gamma ) \sigma \tau < ( 1 - \gamma ) \sigma T+e^{\theta \sigma \tau}(1-\gamma),
\end{split}
\end{equation}
where we use $e^x=1+e^{\theta x}x, ~~0<\theta<1$.

 According to (\ref{2.8}), we can get
 \begin{equation*}
    || \epsilon^{n} || _{\infty} \leq ( 1 - \gamma ) \left(\sigma T+e^{\theta \sigma \tau}\right) r_{max} \leq ( 1 - \gamma ) \left(\sigma T+e^{ \sigma \tau}\right)  C_u(\tau^2+(\Delta x)^2).
 \end{equation*}
 The proof is completed.
  \end{proof}

\subsection{Convergence and stability for 2D  }
In this subsection, the stability and convergence are obtained  by mathematical induction with $L_\infty$ norm.
\begin{theorem} \label{theorem3.3}
The difference scheme (\ref{2.16}) is unconditionally stable.
\end{theorem}
\begin{proof}
Let $ \widetilde u^{n}_{i,j} $ be the approximate solution of $u^{n}_{i,j}$, which is the exact solution of the difference scheme (\ref{2.16}).
Denoting $ \epsilon ^{n}_{i,j} = \widetilde u^{n}_{i,j}- u^{n}_{i,j} $, there exists
\begin{equation*}
\begin{split}
& \left( 1 -\kappa^{\alpha}_{\Delta x,\tau}  g^{\alpha}_{i,i} -\kappa^{\beta}_{\Delta y,\tau}  g^{\beta}_{i,i}\right) \epsilon^{1}_{i,j} -\kappa^{\alpha}_{\Delta x,\tau} \sum ^{N_x}_{m=0,m \neq i} { g^{\alpha}_{i,m} \epsilon^{1}_{m,j}}- \kappa^{\beta}_{\Delta y,\tau} \sum ^{N_y}_{m=0,m \neq j} { g^{\alpha}_{j,m} \epsilon^{1}_{i,m}}\\
&\quad= \epsilon^{0}_{i,j}e^{-\sigma (n-1) \tau}, ~~~~ n=1, \\
& \left( 1 -\kappa^{\alpha}_{\Delta x,\tau}  g^{\alpha}_{i,i} -\kappa^{\beta}_{\Delta y,\tau}  g^{\beta}_{i,i}\right) \epsilon^{n}_{i,j} -\kappa^{\alpha}_{\Delta x,\tau} \sum ^{N_x}_{m=0,m \neq i} { g^{\alpha}_{i,m} \epsilon^{n}_{m,j}}- \kappa^{\beta}_{\Delta y,\tau} \sum ^{N_y}_{m=0,m \neq j} { g^{\alpha}_{j,m} \epsilon^{n}_{i,m}}\\
&\quad=  \sum^{n-1}_{k=1} {\epsilon^{k}_{i,j}} e^{-\sigma(n-1-k)\tau}\left(1-e^{-\sigma \tau}\right) + \epsilon^{0}_{i,j}e^{-\sigma (n-1) \tau}, ~ n>1.
\end{split}
\end{equation*}

We next prove the scheme is unconditionally stable by the mathematical induction.

Let  $ |\epsilon^{n}_{i_0,j_0}|:=||E^{n}||_{\infty} = \max \limits_{0 \leq i \leq N_x,0 \leq j \leq N_y} {|\epsilon^{n}_{i,j}|}$.
From  Lemma \ref{lemma3.3}, we obtain
\begin{equation*}
\begin{split}
&||E^{1}||_{\infty}
= |\epsilon^{1}_{i_0,j_0}| \leq |\epsilon^{1}_{i_0,j_0}| -\kappa^{\alpha}_{\Delta x,\tau} \sum ^{N_x}_{m=0} { g^{\alpha}_{i_0,m} |\epsilon^{1}_{i_0,j_0}}| - \kappa^{\beta}_{\Delta y,\tau} \sum ^{N_y}_{m=0} { g^{\beta}_{j_0,m} |\epsilon^{1}_{i_0,j_0}}| \\
&\leq \left(1 -\kappa^{\alpha}_{\Delta x,\tau} g^{\alpha}_{i_0,i_0}  -\kappa^{\beta}_{\Delta y,\tau} g^{\beta}_{j_0,j_0} \right) |\epsilon^{1}_{i_0}|
         -\kappa^{\alpha}_{\Delta x,\tau} \!\!\!\sum ^{N_x}_{m=0,m \neq i_0} { g^{\alpha}_{i_0,m} |\epsilon^{1}_{m,j_0}}|
         -\kappa^{\beta}_{\Delta y,\tau} \!\!\!\sum ^{N_y}_{m=0,m \neq j_0} { g^{\alpha}_{j_0,m} |\epsilon^{1}_{i_0,m}}|\\
&\leq \left|\left(1 -\kappa^{\alpha}_{\Delta x,\tau} g^{\alpha}_{i_0,i_0}  -\kappa^{\beta}_{\Delta y,\tau} g^{\beta}_{j_0,j_0} \right) \epsilon^{1}_{i_0}
         -\kappa^{\alpha}_{\Delta x,\tau} \sum ^{N_x}_{m=0,m \neq i_0} { g^{\alpha}_{i_0,m} \epsilon^{1}_{m,j_0}}
         -\kappa^{\beta}_{\Delta y,\tau} \sum ^{N_y}_{m=0,m \neq j_0} { g^{\alpha}_{j_0,m} \epsilon^{1}_{i_0,m}}\right| \\
&= |\epsilon^{0}_{i_0,j_0}|e^{-\sigma (n-1) \tau}\leq ||E^{0}||_{\infty}.
\end{split}
\end{equation*}
Assuming $||E^s||_{\infty} \leq ||E^{0}||_{\infty}$, $s = 1, 2,3,\cdots,n-1 $, we have
\begin{equation*}
\begin{split}
&||E^{n}||_{\infty}
= |\epsilon^{n}_{i_0,j_0}| \leq |\epsilon^{n}_{i_0,j_0}| -\kappa^{\alpha}_{\Delta x,\tau} \sum ^{N_x}_{m=0} { g^{\alpha}_{i_0,m} |\epsilon^{n}_{i_0,j_0}}| - \kappa^{\beta}_{\Delta y,\tau} \sum ^{N_y}_{m=0} { g^{\beta}_{j_0,m} |\epsilon^{n}_{i_0,j_0}}| \\
&\leq \left(1 -\kappa^{\alpha}_{\Delta x,\tau} g^{\alpha}_{i_0,i_0}  -\kappa^{\beta}_{\Delta y,\tau} g^{\beta}_{j_0,j_0} \right) |\epsilon^{n}_{i_0}|
         -\kappa^{\alpha}_{\Delta x,\tau}\!\! \sum ^{N_x}_{m=0,m \neq i_0} { g^{\alpha}_{i_0,m} |\epsilon^{n}_{m,j_0}}|
         -\kappa^{\beta}_{\Delta y,\tau} \!\!\sum ^{N_y}_{m=0,m \neq j_0} { g^{\alpha}_{j_0,m} |\epsilon^{n}_{i_0,m}}|\\
&\leq \left|\left(1 -\kappa^{\alpha}_{\Delta x,\tau} g^{\alpha}_{i_0,i_0}  -\kappa^{\beta}_{\Delta y,\tau} g^{\beta}_{j_0,j_0} \right) \epsilon^{n}_{i_0}
         -\kappa^{\alpha}_{\Delta x,\tau} \sum ^{N_x}_{m=0,m \neq i_0} { g^{\alpha}_{i_0,m} \epsilon^{n}_{m,j_0}}
         -\kappa^{\beta}_{\Delta y,\tau} \sum ^{N_y}_{m=0,m \neq j_0} { g^{\alpha}_{j_0,m} \epsilon^{n}_{i_0,m}}\right| \\
& = \left| \sum^{n-1}_{k=1} {\epsilon^{k}_{i_0,j_0}} e^{-\sigma(n-1-k)\tau}\left(1-e^{-\sigma \tau}\right) + \epsilon^{0}_{i_0,j_0}e^{-\sigma (n-1) \tau}\right|\\
&\leq  \left| \sum^{n-1}_{k=1}  e^{-\sigma(n-1-k)\tau}\left(1-e^{-\sigma \tau}\right) +e^{-\sigma (n-1) \tau}\right| ||E^{0}||_{\infty}= ||E^{0}||_{\infty}.
\end{split}
\end{equation*}
The proof is completed.
\end{proof}

\begin{theorem}
Let $u^{n}_{i,j}$ be the approximate solution of $u(x_i,y_j,t_n)$ computed by the difference scheme (\ref{2.16}). Let $\varepsilon ^{n}_{i,j} = u(x_i,y_j,t_n)- u^{n}_{i,j} $. Then
 \begin{equation*}
    || \epsilon^{n} || _{\infty} \leq ( 1 - \gamma ) \left(\sigma T+e^{ \sigma \tau}\right)  C_u\left(\tau^2+(\Delta x)^2+ (\Delta y)^2\right),
 \end{equation*}
where $C_{u}$ is defined by (\ref{2.8}) and $(x_{i},y_{j},t_{n}) \in \Omega \times (0,T]$ with $N\tau \leq T$.
\end{theorem}
\begin{proof}
Let $u(x_{i},y_j,t_{n})$ be the exact solution of (\ref{1.1}) at the mesh point $(x_{i},y_j,t_{n})$. Defined $\varepsilon ^{n}_{i,j} = u(x_i,y_j,t_n)- u^n_{i,j} $.
Subtracting (\ref{2.13}) from (\ref{2.16}) with $\varepsilon ^{0}_{i,j} = 0$, we obtain
\begin{equation*}
\begin{split}
& \left( 1 -\kappa^{\alpha}_{\Delta x,\tau}  g^{\alpha}_{i,i} -\kappa^{\beta}_{\Delta y,\tau}  g^{\beta}_{i,i}\right) \epsilon^{1}_{i,j} -\kappa^{\alpha}_{\Delta x,\tau} \sum ^{N_x}_{m=0,m \neq i} { g^{\alpha}_{i,m} \epsilon^{1}_{m,j}}- \kappa^{\beta}_{\Delta y,\tau} \sum ^{N_y}_{m=0,m \neq j} { g^{\alpha}_{j,m} \epsilon^{1}_{i,m}}\\
&\quad= \frac{(1-\gamma)\sigma \tau}{1-e^{-\sigma \tau}}r^{1}_{i,j}, ~~~~ n=1, \\
& \left( 1 -\kappa^{\alpha}_{\Delta x,\tau}  g^{\alpha}_{i,i} -\kappa^{\beta}_{\Delta y,\tau}  g^{\beta}_{i,i}\right) \epsilon^{n}_{i,j}
   - \kappa^{\alpha}_{\Delta x,\tau} \sum ^{N_x}_{m=0,m \neq i} { g^{\alpha}_{i,m} \epsilon^{n}_{m,j}}
   - \kappa^{\beta}_{\Delta y,\tau} \sum ^{N_y}_{m=0,m \neq j} { g^{\alpha}_{j,m} \epsilon^{n}_{i,m}}\\
&=  \sum^{n-1}_{k=1} {\epsilon^{k}_{i,j}} e^{-\sigma(n-1-k)\tau}\left(1-e^{-\sigma \tau}\right)
  + \frac{(1-\gamma)\sigma \tau}{1-e^{-\sigma \tau}}r^{n}_{i.j}, ~ n>1.
\end{split}
\end{equation*}
We next prove the desired results  by the mathematical induction.

Let  $ |\epsilon^{n}_{i_0,j_0}|:=||\epsilon^{n}||_{\infty} = \max \limits_{0 \leq i \leq N_x,0 \leq j \leq N_y} {|\epsilon^{n}_{i,j}|}$
and $r_{\max} = \max\limits_{0 \leq i \leq N_x,0 \leq j \leq N_y,0 \leq n \leq N}{|r^{n}_{i,j}|}$. Using  Lemma \ref{lemma3.3}, we have
\begin{equation*}
\begin{split}
&||\epsilon^{1}||_{\infty}
= |\epsilon^{1}_{i_0,j_0}| \leq |\epsilon^{1}_{i_0,j_0}| -\kappa^{\alpha}_{\Delta x,\tau} \sum ^{N_x}_{m=0} { g^{\alpha}_{i_0,m} |\epsilon^{1}_{i_0,j_0}}| - \kappa^{\beta}_{\Delta y,\tau} \sum ^{N_y}_{m=0} { g^{\beta}_{j_0,m} |\epsilon^{1}_{i_0,j_0}}| \\
&= |\epsilon^{1}_{i_0,j_0}|-\kappa^{\alpha}_{\Delta x,\tau}{ g^{\alpha}_{i_0,i_0}|\epsilon^{1}_{i_0,j_0}}|
-\kappa^{\alpha}_{\Delta x,\tau} \sum ^{N_x}_{m=0,m \neq i_0} { g^{\alpha}_{i_0,m} |\epsilon^{1}_{i_0,j_0}}|\\
&\quad -\kappa^{\beta}_{\Delta y,\tau}{ g^{\beta}_{j_0,j_0} |\epsilon^{1}_{i_0,j_0}}|
-\kappa^{\beta}_{\Delta y,\tau} \sum ^{N_y}_{m=0,m \neq j_0} { g^{\alpha}_{j_0,m} |\epsilon^{1}_{i_0,j_0}}|\\
&\leq \left(1 -\kappa^{\alpha}_{\Delta x,\tau} g^{\alpha}_{i_0,i_0}  -\kappa^{\beta}_{\Delta y,\tau} g^{\beta}_{j_0,j_0} \right) |\epsilon^{1}_{i_0,j_0}|
         -\kappa^{\alpha}_{\Delta x,\tau} \!\!\!\sum ^{N_x}_{m=0,m \neq i_0} \!\!{ g^{\alpha}_{i_0,m} |\epsilon^{1}_{m,j_0}}|
         -\kappa^{\beta}_{\Delta y,\tau} \!\!\!\sum ^{N_y}_{m=0,m \neq j_0} \!\!{ g^{\alpha}_{j_0,m} |\epsilon^{1}_{i_0,m}}|\\
&\leq \left|\left(1 -\kappa^{\alpha}_{\Delta x,\tau} g^{\alpha}_{i_0,i_0}  -\kappa^{\beta}_{\Delta y,\tau} g^{\beta}_{j_0,j_0} \right) \epsilon^{1}_{i_0,j_0}
         -\kappa^{\alpha}_{\Delta x,\tau} \sum ^{N_x}_{m=0,m \neq i_0} { g^{\alpha}_{i_0,m} \epsilon^{1}_{m,j_0}}
         -\kappa^{\beta}_{\Delta y,\tau} \sum ^{N_y}_{m=0,m \neq j_0} { g^{\alpha}_{j_0,m} \epsilon^{1}_{i_0,m}}\right| \\
&= \left|\frac{(1-\gamma)\sigma \tau}{1-e^{-\sigma \tau}}r^{1}_{i_0,j_0}\right|\leq \frac{(1-\gamma)\sigma \tau}{1-e^{-\sigma \tau}}r_{\max}.
\end{split}
\end{equation*}
Supposing $ || \epsilon^{n}||_{\infty}=|\varepsilon ^{n}_{i_0,j_0}|=\max\limits_{0\leq i\leq N_x,0\leq j\leq N_y}{|\varepsilon ^{n}_{i,j}|}  $ and
$$ ||\epsilon^{s } || _{\infty} \leq \frac{ (s-1) \left(1-e^{- \sigma \tau} \right)+1}{ 1 - e^{- \sigma \tau}}  ( 1 - \gamma ) \sigma \tau  r_{\max},~~ s = 1,2,3,\cdots,n-1,$$
we have
\begin{equation*}
\begin{split}
&||\epsilon^{n}||_{\infty}
= |\epsilon^{n}_{i_0,j_0}| \leq |\epsilon^{n}_{i_0,j_0}| -\kappa^{\alpha}_{\Delta x,\tau} \sum ^{N_x}_{m=0} { g^{\alpha}_{i_0,m} |\epsilon^{n}_{i_0,j_0}}| - \kappa^{\beta}_{\Delta y,\tau} \sum ^{N_y}_{m=0} { g^{\beta}_{j_0,m} |\epsilon^{n}_{i_0,j_0}}| \\
&= |\epsilon^{n}_{i_0,j_0}|-\kappa^{\alpha}_{\Delta x,\tau}{ g^{\alpha}_{i_0,i_0}|\epsilon^{n}_{i_0,j_0}}|
-\kappa^{\alpha}_{\Delta x,\tau} \sum ^{N_x}_{m=0,m \neq i_0} { g^{\alpha}_{i_0,m} |\epsilon^{n}_{i_0,j_0}}|\\
&\quad-\kappa^{\beta}_{\Delta y,\tau}{ g^{\beta}_{j_0,j_0} |\epsilon^{n}_{i_0,j_0}}|
-\kappa^{\beta}_{\Delta y,\tau} \sum ^{N_y}_{m=0,m \neq j_0} { g^{\alpha}_{j_0,m} |\epsilon^{n}_{i_0,j_0}}|\\
&\leq \left(1 -\kappa^{\alpha}_{\Delta x,\tau} g^{\alpha}_{i_0,i_0}  -\kappa^{\beta}_{\Delta y,\tau} g^{\beta}_{j_0,j_0} \right) |\epsilon^{n}_{i_0,j_0}|
         -\kappa^{\alpha}_{\Delta x,\tau} \!\!\!\sum ^{N_x}_{m=0,m \neq i_0}\!\!\! { g^{\alpha}_{i_0,m} |\epsilon^{n}_{m,j_0}}|
         -\kappa^{\beta}_{\Delta y,\tau} \!\!\!\sum ^{N_y}_{m=0,m \neq j_0} \!\!\!{ g^{\alpha}_{j_0,m} |\epsilon^{n}_{i_0,m}}|\\
&\leq \left|\left(1 -\kappa^{\alpha}_{\Delta x,\tau} g^{\alpha}_{i_0,i_0}  -\kappa^{\beta}_{\Delta y,\tau} g^{\beta}_{j_0,j_0} \right) \epsilon^{n}_{i_0,j_0}
         -\kappa^{\alpha}_{\Delta x,\tau} \sum ^{N_x}_{m=0,m \neq i_0} { g^{\alpha}_{i_0,m} \epsilon^{n}_{m,j_0}}
         -\kappa^{\beta}_{\Delta y,\tau} \sum ^{N_y}_{m=0,m \neq j_0} { g^{\alpha}_{j_0,m} \epsilon^{n}_{i_0,m}}\right| \\
& = \left| \sum^{n-1}_{k=1} {\epsilon^{k}_{i_0,j_0}} e^{-\sigma(n-1-k)\tau}\left(1-e^{-\sigma \tau}\right) + \epsilon^{0}_{i_0,j_0}e^{-\sigma (n-1) \tau}\right| \leq \Phi_\tau \cdot r_{\max},
\end{split}
\end{equation*}
where $\Phi_\tau$ is given by (\ref{3.004}).

 According to (\ref{2.14}), we can get
 \begin{equation*}
    || \epsilon^{n} || _{\infty} \leq ( 1 - \gamma ) \left(\sigma T+e^{\theta \sigma \tau}\right) r_{max}
    \leq ( 1 - \gamma ) \left(\sigma T+e^{ \sigma \tau}\right)  C_u\left(\tau^2+(\Delta x)^2+(\Delta y)^2\right).
 \end{equation*}
 The proof is completed.
  \end{proof}

\section{Numerical results}
In the section, we numerically verify the above theoretical results including convergence orders and numerical stability.  And the $l_\infty$ norm and the discrete  $L^2$-norm, respectively,  are used to measure the numerical errors.

\begin{example}\end{example}
Consider (\ref{2.0016}) on a finite domain with $0<x <1$,  $0<t \leq 1$,  and the forcing function is
\begin{equation*}
\begin{split}
   f(x,t)
    =& -\frac {\sigma} {1-\gamma} t e^{- \sigma t} x^{2} (1-x)^{2} + \frac{e^ {- \sigma t}}{2 \cos (\alpha \pi /2)} \frac{24}{\Gamma(5- \alpha)} (x^{4- \alpha} + (1-x)^{4- \alpha}) \\
   & - \frac{ e^ {- \sigma t} }{2 \cos (\alpha \pi /2)} \frac{12}{\Gamma(4- \alpha)} (x^{3- \alpha} + (1-x)^{3- \alpha}) \\
   & + \frac{ e^ {- \sigma t} }{2 \cos (\alpha \pi /2)} \frac{2}{\Gamma(3- \alpha)} (x^{2- \alpha} + (1-x)^{2- \alpha}),~~ \sigma=\frac{\gamma}{1-\gamma}
\end{split}
\end{equation*}
with the nonzero initial condition $ u(x,0) = x^{2}(1-x)^{2} $ and the homogeneous Dirichlet boundary conditions. The exact solution of the fractional PDEs  is
$$ u(x,t) = e^{- \sigma t } x^{2} (1-x)^{2}. $$

\begin{table}[h]\fontsize{9.5pt}{12pt}\selectfont
 \begin{center}
  \caption {The maximum errors and convergence orders for  (\ref{2.10}) with $\tau=\Delta x $.} \vspace{5pt}
\begin{tabular*}{\linewidth}{@{\extracolsep{\fill}}*{8}{c}}                                    \hline  
$l_\infty$ norm       &    $\tau$   & $\alpha=1.2$  &  Rate       & $\alpha=1.8$ &   Rate   \\\hline
                      &  ~~1/40  &  1.0686e-04   &             & 1.3426e-04   &            \\
  $ \gamma = 0.1 $    &  ~~1/80  &  2.9917e-05   &  1.8367     & 3.3543e-05   & 2.0009      \\
                      &  ~~1/160 &  7.9022e-06   &  1.9206     & 8.3559e-06   & 2.0051       \\
                      &  ~~1/320 &  2.0766e-06   &  1.9281     & 2.0766e-06   & 2.0086        \\\hline 
                      &  ~~1/40  &  4.4671e-05   &             & 6.0820e-05   &             \\
  $ \gamma = 0.5 $    &  ~~1/80  &  1.2415e-05   &  1.8473     & 1.5196e-05   & 2.0009       \\
                      &  ~~1/160 &  3.2730e-06   &  1.9234     & 3.7834e-06   & 2.0059        \\
                      &  ~~1/320 &  8.6348e-07   &  1.9224     & 9.3952e-07   & 2.0097         \\\hline 
                      &  ~~1/40  &  2.9977e-05   &             & 6.8820e-06   &             \\
  $ \gamma = 0.9 $    &  ~~1/80  &  7.4790e-06   &  2.0029     & 1.7237e-06   & 1.9973       \\
                      &  ~~1/160 &  1.8634e-06   &  2.0049     & 4.3057e-07   & 2.0012        \\
                      &  ~~1/320 &  4.6419e-07   &  2.0051     & 1.0735e-07   & 2.0040         \\\hline 
    \end{tabular*}\label{table:4.1}
  \end{center}
\end{table}
Tables  \ref{table:4.1} and \ref{table:4.2}  show that the schemes (\ref{2.10}) have the global truncation errors $\mathcal{O} (\tau^2+(\Delta x)^2)$  at time $T=1$.
Here the $l_\infty$ norm and the discrete  $L^2$-norm, respectively,  are used to measure the numerical errors  for   (\ref{2.10})  with  $\tau=\Delta x$.

\begin{table}[h]\fontsize{9.5pt}{12pt}\selectfont
 \begin{center}
  \caption {The discrete $L^2$-norm errors and convergence orders for  (\ref{2.10}) with $\tau=\Delta x $.} \vspace{5pt}
\begin{tabular*}{\linewidth}{@{\extracolsep{\fill}}*{8}{c}}                                    \hline  
$L^2$-norm       &    $\tau$   & $\alpha=1.2$  &  Rate       & $\alpha=1.8$ &   Rate   \\\hline
                      &  ~~1/40  &  6.6304e-05   &             & 8.9805e-05   &            \\
  $ \gamma = 0.1 $    &  ~~1/80  &  1.6925e-05   &  1.9700     & 2.2274e-05   & 2.0114      \\
                      &  ~~1/160 &  4.3290e-06   &  1.9670     & 5.5195e-06   & 2.0128       \\
                      &  ~~1/320 &  1.1060e-06   &  1.9687     & 1.3670e-06   & 2.0135        \\\hline 
                      &  ~~1/40  &  3.0386e-05   &             & 4.0823e-05   &             \\
  $ \gamma = 0.5 $    &  ~~1/80  &  7.6617e-06   &  1.9877     & 1.0123e-05   & 2.0118       \\
                      &  ~~1/160 &  1.9408e-06   &  1.9810     & 2.5056e-06   & 2.0144        \\
                      &  ~~1/320 &  4.9221e-07   &  1.9793     & 6.1953e-07   & 2.0159         \\\hline 
                      &  ~~1/40  &  2.2186e-05   &             & 4.9221e-06   &             \\
  $ \gamma = 0.9 $    &  ~~1/80  &  5.5359e-06   &  2.0027     & 1.2327e-06   & 1.9975       \\
                      &  ~~1/160 &  1.3791e-06   &  2.0050     & 3.0788e-07   & 2.0014        \\
                      &  ~~1/320 &  3.4350e-07   &  2.0054     & 7.6749e-08   & 2.0042         \\\hline 
    \end{tabular*}\label{table:4.2}
  \end{center}
\end{table}

\begin{example} \end{example}
Consider (\ref{1.1}) on a finite domain with $0<x <1$, $0<y<1$, $0<t \leq 1$, and the forcing function is
\begin{equation*}
\begin{split}
   f(x,y,t)
    =& -\frac {\sigma} {1-\gamma} t e^{- \sigma t} x^{2} (1-x)^{2}y^{2} (1-y)^{2}+\frac{e^ {- \sigma t}y^{2} (1-y)^{2}}{2 \cos (\alpha \pi /2)}\frac{24(x^{4- \alpha} + (1-x)^{4- \alpha})}{\Gamma(5- \alpha)}  \\
    &+ \frac{e^ {- \sigma t}}{2 \cos (\alpha \pi /2)} y^{2} (1-y)^{2}\left(\frac{2 (x^{2- \alpha} + (1-x)^{2- \alpha})}{\Gamma(3- \alpha)} - \frac{12 (x^{3- \alpha} + (1-x)^{3- \alpha})}{\Gamma(4- \alpha)}  \right)\\
   &+ \frac{e^ {- \sigma t}}{2 \cos (\alpha \pi /2)} x^{2} (1-x)^{2} \left(\frac{24(y^{4- \alpha} + (1-y)^{4- \alpha})}{\Gamma(5- \alpha)} -\frac{12(y^{3- \alpha} + (1-y)^{3- \alpha})}{\Gamma(4- \alpha)} \right) \\
   & + \frac{ e^ {- \sigma t} }{2 \cos (\alpha \pi /2)}x^{2} (1-x)^{2} \frac{2}{\Gamma(3- \alpha)} (y^{2- \alpha} + (1-y)^{2- \alpha}),~~ \sigma=\frac{\gamma}{1-\gamma}
\end{split}
\end{equation*}
with the nonzero initial condition $u(x,y,0)=x^{2} (1-x)^{2}y^{2} (1-y)^{2}$ and the homogeneous Dirichlet boundary conditions.
The exact solution of the (\ref{1.1}) is
$$u(x,y,t)=e^{-\sigma t}x^{2} (1-x)^{2}y^{2} (1-y)^{2}.$$

\begin{table}[h]\fontsize{9.5pt}{12pt}\selectfont
 \begin{center}
  \caption {The  maximum  errors and convergence orders for  (\ref{2.16}) with $\tau=\Delta x=\Delta y$.} \vspace{5pt}
\begin{tabular*}{\linewidth}{@{\extracolsep{\fill}}*{8}{c}}                                    \hline  
$l_\infty$ norm        &    $\tau$   & $\alpha=1.2,\beta=1.3$  &  Rate       & $\alpha=1.8,\beta=1.7$ &   Rate    \\\hline
                      &  ~~1/10  &  8.7959e-05   &             & 1.0126e-04   &            \\
  $ \gamma = 0.3 $    &  ~~1/20  &  2.1543e-05   &  2.0296     & 2.5618e-05   & 1.9828      \\
                      &  ~~1/40  &  5.2815e-06   &  2.0282     & 6.4708e-06   & 1.9851       \\
                      &  ~~1/80  & 1.3016e-06    &  2.0207     & 1.6713e-06   &  1.9530      \\\hline 
                      &  ~~1/10  &  2.2733e-05   &             & 1.8114e-05   &             \\
  $ \gamma = 0.7 $    &  ~~1/20  &  5.5809e-06   &  2.0262     & 4.5869e-06   & 1.9815       \\
                      &  ~~1/40  &  1.3680e-06   &  2.0284     & 1.1568e-06   & 1.9873        \\
                      &  ~~1/80  &  3.4094e-07   &  2.0045     & 2.9480e-07   & 1.9723         \\\hline 
    \end{tabular*}\label{table:4.3}
  \end{center}
\end{table}

Tables  \ref{table:4.3}  shows that the maximum error, at time  $T=1$ and $\tau=\Delta x=\Delta y$, between the exact analytical value and the numerical value.
The scheme  (\ref{2.16}) is second-order convergence and this is in agreement with the order of the truncation error.

\section{Conclusions}
As is well known,  there is less than the second-order convergence for  the  Caputo fractional derivative \cite{Lin:07,Oldham:74} with L1 formula.
We notice that there are already some theoretical convergence results for Caputo-Fabrizio fractional derivative \cite{Akman:18,Atangana:16} with L1 formula.
However, it seems that achieving a second-order accurate scheme (optimal estimates)  is not an easy task.
To  our knowledge, this is the first published finite difference method to consider the space fractional   diffusion  equations with the time   Caputo-Fabrizio fractional derivative.
The optimal estimates with the second-order convergence for L1 scheme are given by two methods.
We remark that the corresponding theoretical including  a prior estimate can also be extended to the nonzero initial values \cite{CSD:18,Ji:15}.

\section*{Acknowledgements}

This work was supported by NSFC 11601206.

\bibliographystyle{amsplain}

\begin{thebibliography}{10}


\bibitem{Abdulhameed:19}
Abdulhameed,  M., Muhammad,  M.M.,  Gital, A.Y., Yakubu, D.G., Khan, I.:
Effect of fractional derivatives on transient MHD flow and radiative heat transfer in a micro-parallel channel at high zeta potentials. Phys. A. \textbf{484},  42-71 (2019).

\bibitem{Abdulhameed:17} Abdulhameed,  M., Vieru, D.,  Roslan, R.:
Magnetohydrodynamic electroosmotic flow of Maxwell fluids with Caputo-Fabrizio derivatives through circular tubes. Comput. Math. Appl. \textbf{74},  2503-2519 (2017).


\bibitem{Akman:18}
Akman, T., Yildiz, B., Baleanu, D.:
New discretization of Caputo-Fabrizio derivative.
Comp. Appl. Math. \textbf{37}, 3307-3333 (2018).




\bibitem{Alkahtani:16}
Alkahtani, B.S.T., Atangana, A.:
Controlling the wave movement on the surface of shallow water with the Caputo-Fabrizio derivative with fractional order.
Chaos. Soliton. Fract. \textbf{89},  539-546  (2016).

\bibitem{Asif:18}
Asif, N.A.,   Hammouch, Z., Riaz, M.B.,  Bulut, H.:
Analytical solution of a Maxwell fluid with slip effects in view of the Caputo-Fabrizio derivative.
Eur. Phys. J. Plus \textbf{133}:272  (2018).




 \bibitem{Atanackovie:18}
Atanackovi\'e, T.M., Pilipovi\'e, S., Zorica, D.:
Properties of the Caputo-Fabrizio fractional derivative and its distributional settings.
Fract. Calc. Appl. Anal. \textbf{21},  29-44  (2018).


 \bibitem{Atangana:16} Atangana,  A.,  Alqahtani, R.T.:
Numerical approximation of the space-time Caputo-Fabrizio fractional derivative and application to groundwater pollution equation.
Adv. Differ. Equ. \textbf{2016},  1-13 (2016).


%
\bibitem{Atangana:18}
Atangana, A., Alkahtani, B.S.T.:
New model of groundwater flowing within a confine aquifer: application of Caputo-Fabrizio derivative.
Arab. J. Geosci. \textbf{9}, 1-6 (2016).

\bibitem{Caputo:15}
Caputo,M., Fabrizio, M.:
A new definition of fractional derivative without singular kernel.
Progr. Fract. Differ. Appl. \textbf{1},  73-85  (2015).


%
%

\bibitem{CD:14}  Chen, M.H.,  Deng, W.H.: Fourth order accurate scheme for the space fractional diffusion equations.
SIAM J. Numer. Anal. \textbf{ 52},    1418-1438  (2014).



\bibitem{CD:15}
Chen, M.H.,  Deng, W.H.: Discretized fractional substantial calculus.
 ESAIM: M2AN. \textbf{49},  373-394 (2015).


 \bibitem{Chen:014}
 Chen, M.H.,  Deng, W.H.:
 A second-order numerical method for two-dimensional two-sided space fractional convection diffusion equation.
 Appl. Math. Model.  \textbf{38}  3244-3259  (2014).



 \bibitem{CSD:18}
 Chen, M.H.,   Shi,J.K.,  Deng, W.H.:
 High order algorithms for Fokker-Planck equation with Caputo-Fabrizio fractional derivative.
arXiv:1809.03263.

 \bibitem{Chen:14}
 Chen, M.H.,   Wang, Y.T.,    Cheng, X.,  Deng, W.H.:
 Second-order LOD multigrid method for multidimensional Riesz fractional diffusion equation.
 BIT.  \textbf{54} 623-647  (2014).



\bibitem{Djida:17}
 Djida, J.D.,  Atangana, A.:
More generalized groundwater model with space-time Caputo Fabrizio fractional differentiation.
Numer. Meth. Part. D. E. \textbf{33}   1616-1627 (2017).


 \bibitem{Firoozjaee:18}
Firoozjaee, M.A., Jafari, H.,  Lia, A.,  Baleanu, D.:
Numerical approach of Fokker-Planck equation with Caputo-Fabrizio fractional derivative using Ritz approximation.
J. Comput. Appl. Math. \textbf{339}  367-373  (2018).

\bibitem{Ji:15}
 Ji,  C.C.,    Sun, Z.Z.:  A high-order compact finite difference schemes for the fractional sub-diffusion equation.
J. Sci. Comput.  \textbf{64},   959-985 (2015).










\bibitem{Lin:07}
Lin, Y.M., Xu, C.J.:
Finite difference/spectral approximations for the time-fractional diffusion equation.
J. Comput. Phys. \textbf{225}, 1533-1552  (2007).

\bibitem{Liu:17}
Liu, Z.G., Cheng, A.J., Li, X.L.:
A second order Crank-Nicolson scheme for fractional Cattaneo equation based on new fractional derivative.
Appl. Math. Comput. \textbf{311}, 361-374 (2017).



\bibitem{Loh:18}
Loh,J.R., Jafari,  H., Isah,  A., Phang, C., Toh, Y.T.:
On the new properties of Caputo-Fabrizio operator and its application in deriving shifted Lagendre operational matrix.
Appl. Numer.  Math. \textbf{132}, 138-153 (2018).


\bibitem{Lubich:86} Lubich, Ch.: Discretized fractional calculus.
SIAM J. Math. Anal. \textbf{17}, 704-719  (1986).

\bibitem{Mahsud:18}
 Mahsud, Y., Shah,  N.A., Vieru, D.:
 Natural convection flows and heat transfer with exponential memory of a Maxwell fluid with damped shear stress.
Comput. Math. Appl. \textbf{76},  2246-2261 (2018).

\bibitem{Oldham:74} Oldham, K.,   Spanier, J.: The Fractional Calculus: Theory and Applications of Differentiation and Integration to Arbitrary Order.
Academic Press,  (1974).

\bibitem{Podlubny:99}  Podlubny, I.:  Fractional Differential Equations.  Academic Press, (1999).



\bibitem{Quarteroni:07}
Quarteroni, A.,  Sacco,  R., Saleri,  F.:
 Numerical Mathematics.
Springer, (2007).



\bibitem{Shah:16}
Shah, N.A., Khan, I.:
Heat transfer analysis in a second grade fluid over and oscillating vertical plate using fractional Caputo-Fabrizio derivatives.
Eur. Phys. J. C \textbf{76}:362  (2016).





\bibitem{Ullah:18}
Ullan, S.,   Khan, M.A., Farooq, M.:
A new fractional model for the dynamics of the hepatitis B virus using the Caputo-Fabrizio derivative.
Eur. Phys. J. Plus  \textbf{133}:237  (2018).

\bibitem{Zhao:14}
Zhao, X.,  Sun, Z.Z.: Compact Crank-Nicolson schemes for a class of fractional Cattaneo equation in inhomogeneous medium. J. Sci. Comput.
\textbf{62}, 747-771 (2014).








\end{thebibliography}

\end{document}